\documentclass{amsart}
\usepackage{amsmath}
\usepackage{amsfonts}
\usepackage{amssymb}
\usepackage{amsthm}
\usepackage{setspace}
\usepackage{graphicx}
\usepackage{wasysym}
\usepackage{listings}
\usepackage{mathrsfs}
\usepackage{tikz}

\def\bfa{{\bf a}}

\def\Hdim{\operatorname{Hdim}}
\def\HSpec{\operatorname{HSpec}}
\def\ConSpan{\operatorname{Conv}}
\def\hgt{\operatorname{hgt}}
\def\condeg{\operatorname{d}_{\operatorname{conv}}}

\def\vsemifield0{$\nu$-semifield$^\dagger$}
\def\vsemifields0{$\nu$-semifields$^\dagger$}
\def\vdomain0{$\nu$-domain$^\dagger$}
\def\vdomains0{$\nu$-domains$^\dagger$}

\def\scrL{\mathscr L}

\def\onenu{1^\nu}

\def\Frac{{\operatorname{Frac}}}
\def\FF{{\langle F \rangle}}

\theoremstyle{plain}
\newtheorem{thm}{Theorem}[section]

\newtheorem{exampl}[thm]{Example}
\newtheorem{lem}[thm]{Lemma}
\newtheorem{cor}[thm]{Corollary}
\newtheorem{prop}[thm]{Proposition}
\newtheorem{constr}[thm]{Construction}
\theoremstyle{remark}
\newtheorem{rem}[thm]{Remark}
\newtheorem{note}[thm]{Note}

\theoremstyle{definition}
\newtheorem{defn}[thm]{Definition}

\newtheorem{exmp}[thm]{Example}

\numberwithin{equation}{section}

\DeclareMathOperator{\PCon}{\mathcal P}

\renewcommand{\Im}{\operatorname{Im\,}}

\def\la{\lambda}

\def\Ker{\mathcal Kern}

\def\semiring0{semiring$^\dagger$}
\def\Semirings0{Semiring$^\dagger$}
\newcommand{\etype}[1]{\renewcommand{\labelenumi}{(#1{enumi})}}
\def\eroman{\etype{\roman}}

\def\semialg0{semi-algebra$^\dagger$}
\def\semirings0{semirings$^\dagger$}
\def\domain0{domain$^\dagger$}
\def\domains0{domains$^\dagger$}
\def\field0{semifield$^\dagger$}
\def\semifield0{semifield$^\dagger$}
\def\Semifield0{Semifield$^\dagger$}
\def\semifields0{semifields$^\dagger$}
\def\Semifields0{Semifields$^\dagger$}
\def\Skel{\operatorname{1}_{\operatorname{loc}}}

\def\tG{\mathcal{G}}

\def\ptM{M}

\def\tT{\mathcal{T}}

\def\one{\mathbf{1}}

\def\rone{\one_R}

\def\skel{{\operatorname{1-set}}}

\def\a{\alpha}

\def\nucong{\cong_\nu}

\def\Cong{\Omega}

\def\fone{\one_F}

\def\Fun{{\operatorname{Fun}}}
\def\FunF{\Fun (F^{(n)},F)}

\begin{document}
\pagenumbering{arabic}

\title{A tropical Krull-Schmidt theorem}%
\author{Tal Perri, Louis Rowen}
\email {talperri@zahav.net.il} \email{rowen@math.biu.ac.il}
\maketitle

\begin{abstract} Continuing the study of semifield kernels,
 We   develop  some algebraic structure notions such as composition series and convexity degree, along with some notions holding a geometric interpretation, like reducibility and hyperdimension.
\end{abstract}

\tableofcontents
\clearpage

\section{Overview}

 Continuing the study of semifield kernels in tropical mathematics
initiated in the doctoral dissertation of the first author and
\cite{AlgAspTropMath,Kern}, we turn to the basic question of a
tropcial Krull-Schmidt theory. Throughout, $F$ denotes a
$\nu$-archimedean \vsemifield0  (to be defined shortly),
which from Proposition~\ref{prop_convex_dependence_of_HP_property} onwards is assumed to be divisible,
 and
$F(\Lambda)$ is the \vsemifield0 of fractions of the polynomial
\semiring0 $F[\Lambda]$ in the indeterminates $\Lambda = \{ \la_1,
\dots, \la _n\}$.

The most intuitive way to develop algebraic geometry over a semiring
would be to consider the coordinate semiring of a variety,
cf.~\cite{GG,MacRin, Coordinate}. But unfortunately homomorphisms of
semirings are defined by congruences, not ideals, and the study of
congruences is much more difficult than that of ideals. This led the
authors in \cite{Kern} to look for an alternative algebraic
structure, namely that of semifield kernels.

 As seen via   \cite[Chapter 6]{Kern} and in particular
Theorem~\ref{cor11}, tropical varieties correspond to coordinate
$\nu$-\semifields0 (Definition~\ref{def_coord_semifield_1}), and
thus, in view of Proposition~\ref{prop_coord_semifield_1}, chains of
kernels of $F(\Lambda)$, in particular, HP-kernels give us an
algebraic tropical notion of dimension. Its determination is the
subject of this paper, in which we focus on irreducible HP-kernels,
which comprise the \textbf{hyperspace-spectrum},
cf.~Definition~\ref{hyperspec}.  We get to the main results in
\S\ref{compcon}, using standard    Krull-Schmidt theory, proving
catenarity in Theorem~\ref{caten}, and concluding that $F(\Lambda)$
has dimension $n$ in
Theorem~\ref{prop_kernel_descending_chain_properties} and
Corollary~\ref{correctdim}.

Our tools include convexity degree, along with some notions having
a geometric interpretation, such as reducibility and hyperdimension.

\section{Background}

We recall the main ideas of \cite{Kern}, starting with a general
review.

\subsection{Semirings without zero} $ $

 \begin{defn}\label{semir1} A
\textbf{\semiring0} (semiring without zero)  is a
set~$R:=(R,+,\cdot,1)$ equipped with
 binary operations $+$ and~$\cdot \;$ and distinguished element $\rone$ such that:
\begin{enumerate}\eroman
    \item $(R, + )$ is an Abelian semigroup;
    \item $(R, \cdot \ , \rone )$ is a monoid with identity element
    $\rone$;
    \item Multiplication distributes over addition.
      \item $R$ contains elements $r_0$ and $r_1$ with $r_0+r_1 = \rone.$

\end{enumerate}
\end{defn}


A \textbf{\domain0} is a commutative \semiring0 whose multiplicative
monoid is cancellative.

\begin{defn} A
\textbf{\semifield0} is a \domain0 in which every element is
(multiplicatively) invertible.
\end{defn}

 (In other words, the multiplicative
monoid is an Abelian group.) We need a fundamental correspondence
between ordered monoids and \semirings0,
\cite[\S4]{Semifields_LO_Groups}:

\begin{rem}\label{corresp} Any \semiring0 can be viewed as a (multiplicative) semi-lattice
ordered Abelian monoid, where we define
\begin{equation}\label{add0} a \vee b: = a+b .\end{equation}
Thus, we have a natural partial order given by $a \ge b$ whenever $a
= b+c$ for some $c$. (This partial order is trivial for rings, but
not for idempotent semirings!)

Conversely, any semi-lattice ordered Abelian monoid $\ptM $ becomes
a \semiring0, where multiplication is the given monoid operation and
addition is given by
\begin{equation}\label{add0} a+b : = a \vee b\end{equation}
(viewed in $\ptM$).\end{rem}

\subsection{Supertropical \vsemifields0}$ $

We bring in the ``ghost'' notation.

\begin{defn}\label{vdom}\cite[Definition~3.2.1]{Kern} A \textbf{\vdomain0} is a quadruple
$(R, \tT, \nu, \tG)$ where  $R$ is a \semiring0 and $\tT \subset R$
is a cancellative multiplicative submonoid  and $\tG \triangleleft
R$ is endowed with a partial order, together with an idempotent
homomorphism $\nu : R \to \tG$, with $\nu |_\tT$ onto, satisfying
the conditions:
$$  a+b =  a  \quad \text{whenever} \quad \nu (a) > \nu (b). $$
$$ a+b = \nu (a) \quad \text{whenever} \quad \nu (a) = \nu (b). $$
 $\tT$ is called  the \textbf{tangible
submonoid} of $R$.    $ \tG$ is called the \textbf{ghost ideal}.

\end{defn}
%
%
%
%
%
%

We write $a^\nu$ for $\nu(a),$ for $a\in R.$ We write $a \nucong
  b$ if $a^\nu = b^\nu,$ and say that $a$ and $b$ are
  $\nu$-\textbf{equivalent}. Likewise we write $a \ge_\nu b$ (resp.~ $a >_\nu
  b$) if $a^\nu \ge  b^\nu$ (resp.~ $a^\nu >  b^\nu$).

%

 \begin{defn}\label{vsemi2}
A \textbf{$\nu$-\semifield0} is a
  \vdomain0 for which the tangible
submonoid $\tT$ is an Abelian group.
A \textbf{supertropical
\semifield0} is a $\nu$-\semifield0 $F = (F, \tT, \nu, \tG)$  for
which $F = \tT \cup \tG$ and $\tG$ is totally ordered, a special
case of~\cite{SuperTropicalAlg}.
\end{defn}

%
%
%

\begin{exampl}\label{stropi} Given a monoid $\ptM$ and an ordered group $\tG$
with an isomorphism $\nu: \ptM \to \tG,$ we write $a ^\nu$ for
$\nu(a).$ The \textbf{standard supertropical} monoid $R$ is the
disjoint union $ \tT \cup \tG$ where $\tT$ is taken to be  $\ptM$,
made into a monoid by starting with the given products on $\ptM$ and
$\tG$, and defining $ a b^\nu $ and $ a^\nu b $ to be $ (ab)^\nu$
for $a,b \in \ptM$.

We extend $v$ to the \textbf{ghost map} $\nu: R\to \tG$ by taking
$\nu|_\ptM = v$ and  $\nu_\tG$ to be the identity on $ \tG$. Thus,
$\nu$ is a monoid projection.

We make $R$ into a \semiring0, called the \textbf{standard
supertropical \semifield0}, by defining $$a+b =
\begin{cases} a \text{ for } a >_\nu b;\\b \text{ for } a
<_\nu b;\\a^\nu \text{ for } a\nucong b.
\end{cases}$$

 $R$ is never
additively cancellative, since $$a + a^\nu = a^\nu = a^\nu +
a^\nu.$$
%
%

\end{exampl}
%
%
%

\subsubsection{$\nu$-Localization}$ $

 If $R=(R,\tT, \tG,\nu)$ is a  \vdomain0, then we call $ \tT ^{-1}R $ the
\textbf{$\nu$-\semifield0 of fractions} $\Frac _\nu R$ of $R$.

\begin{lem}\label{vsemi20}
  $\Frac_\nu R$
 is a \vsemifield0 in the obvious way.
\end{lem}
\begin{proof} Define $\nu( \frac rs) = \frac{r^\nu}s.$
\end{proof}

 \subsection{Kernels of \Semirings0}$ $

 The role of ideals is
replaced here by kernels.

\begin{defn} A \textbf{kernel} of a \semiring0 $\mathcal{S}$ is a
subgroup ${K}$ which is \textbf{convex} in the sense that if  $a,b
\in {K}$ and $\a, \beta \in \mathcal{S}$ with $\a + \beta =\fone,$
then $\a a + \beta b \in {K}.$\end{defn}

\begin{prop}\label{basicprop}\cite[Proposition~4.1.3]{Kern} If $\Cong$ is a congruence
on a \semifield0 $\mathcal{S}$, then ${K}_\Cong = \left\{ a  \in
\mathcal{S} : a \equiv 1\right\}$ is a kernel. Conversely, any
kernel $K$ of $\mathcal{S}$ defines a congruence according to
\cite[Definition~3.1]{Hom_Semifields}, i.e., $a \equiv b$ iff $\frac
a b \equiv 1.$ If $\mathcal{S}$ is the \semifield0 of the
lattice-ordered group $G$, then the \semifield0 $\mathcal{S}/\rho_K$
is the \semifield0 of the lattice-ordered group~$G/K.$
\end{prop}

\begin{rem}\label{bfacts}\cite[Remark~4.1.4]{Kern}$ $

\begin{enumerate} \eroman

%

 \item \cite[Corollary~1.1]{Cont_Semifields}, \cite[Property~2.4]{Prop_Semifields} Any
 kernel ${K}$ is convex with respect to the order of Remark~\ref{corresp},
 in the sense that if $a \le b \le c$ with $a,c \in {K}$,
then $b \in {K}.$

 \item \cite[Proposition~2.3]{Cont_Semifields}. If $|a| \in {K}$, a kernel,
then $a\in {K}$.

\item \cite{Cont_Semifields} The product $K_1   K_2 = \{ ab \ : \
a \in K_1, b \in K_2 \}$   of two kernels  is    a kernel, in fact
the smallest kernel containing $K_1 \cup K_2$.

 \item The intersection of   kernels is a kernel. Thus, for any set $S \subset \mathcal{S}$ we
can define the kernel $\langle S \rangle$  \textbf{generated by} $S$
to be the intersection of all kernels containing $S$.

\item \cite[Theorem~3.5]{Cont_Semifields}. Any kernel generated by
a finite set $\{ s_1, \dots, s_m\}$ is in fact generated by the
single element $\sum _{i=1}^m (s_i + s_i^{-1}).$

\item The kernel generated by $a \in \mathcal{S}$ is just the set
of finite sums $\{  \sum_i  b_i  a^i :\ b_i \in \mathcal{S}, \ \sum
b_i = 1\}.$

\item  \cite[Theorem~3.8]{Hom_Semifields}. If ${K}$ is a kernel of
a \semifield0 $\mathcal{S}$ and the \semifield0 $\mathcal{S}/{K}$ is
idempotent, then ${K}$ is a sub-\semifield0 of $\mathcal{S}$. (This
is because for $a,b \in {K}$ the image of $a+b$ is $1{K} + 1{K} =
1{K}.$)

\item Let $K$ be a kernel of a \semifield0 $\mathcal{S}$. For
every $a \in \mathcal{S}$, if $a^{n} \in K$ for some $n \in
\mathbb{N}$ then $a \in K$.
 \item The kernel of a kernel is a
kernel.
\end{enumerate}
\end{rem}

We also need the following generalization of (vi):
\begin{prop}\label{prop_ker_stracture_by group}\cite[Proposition (3.13)]{Hom_Semifields}
Let $\mathbb{S}$ be a semifield and let $N$ be a (normal) subgroup of $( \mathbb{S}, \cdot)$.
Then the smallest kernel containing $N$ is
\begin{equation}
 \left\{ \sum_{i=1}^{n}s_ih_i \ : \ n \in \mathbb{N}, \ h_i \in  N, \ s_i \in \mathbb{S} \ \text{such that} \ \sum_{i=1}^{n}s_i = 1 \right\}.
\end{equation}
\end{prop}

Next, we recall \cite[Theorem~4.1.6 ff.]{Kern}, which really is a
special case of the basic lattice correspondence from universal
algebra:

\begin{thm}\label{latticecor}
Let $\phi : \mathcal{S}_1 \rightarrow \mathcal{S}_2$ be a
\semifield0 homomorphism. Then the following hold:
\begin{enumerate}
 \item For any kernel $L$ of \ $\mathcal{S}_1$, $\phi(L)$ is a kernel of \ $\phi(\mathcal{S}_1)$.
  \item For a kernel $K$ of \ $\phi(\mathcal{S}_1)$, $\phi^{-1}(K)$ is a kernel of \ $\mathcal{S}_1$. In particular, for any kernel $L$ of \ $\mathcal{S}_1$, $\phi^{-1}(\phi(L)) = K\L$ is a kernel of \ $\mathcal{S}_1$.
\end{enumerate}
\end{thm}

In particular, $\phi^{-1}(1)$ is a kernel.

\begin{cor}\label{cor_subdirect_decomposition_for_semifield_of_fractions}
There is an injection $\mathbb{S}/( K_1 \cap K_2 \cap \dots  \cap
K_t) \hookrightarrow \prod_{i=1}^{t} \mathbb{S}/K_i$, for any
kernels $K_i$  of a  \semifield0~$\mathbb{S}$, induced by the map $f
\mapsto (fK_i).$
\end{cor}

We also have the isomorphism theorems.

\begin{thm}\label{thm_nother_1_and_3}
Let $\mathcal{S}$ $K,L$ be kernels of $\mathcal{S}$.
\begin{enumerate}
  \item If $\mathcal{U}$ is a sub-\semifield0 of  $\mathcal{S}$, then $\mathcal{U} \cap K$ is a kernel of $\mathcal{U}$, and $K$ a kernel of the sub-\semifield0  $\mathcal{U}  K = \{u \cdot k \ : \ u \in \mathcal{U}, \ k \in K \}$ of \ $\mathcal{S}$,
   and one has the isomorphism $$\mathcal{U}/(\mathcal{U} \cap K) \cong \mathcal{U} K/K.$$
  \item $L \cap K$ is a kernel of $L$ and $K$ a kernel of $L  K$, and the group isomorphism $$L/(L \cap K) \cong L   K/K$$   is a \semifield0 isomorphism.
  \item If $L \subseteq K$, then $K/L$ is a kernel of \ $\mathcal{S}/L$ and one has the \semifield0 isomorphism $$\mathcal{S}/K \cong (\mathcal{S}/L)/(K/L).$$
\end{enumerate}

Let $L $ be a kernel of  a \semifield0 $\mathcal{S}$. Every kernel
of $\mathcal{S}/L$ has the form $K/L$ for some uniquely determined
kernel $K \supseteq L$, yielding a lattice isomorphism
$$ \{ \text{Kernels of} \ \mathcal{S}/L \} \rightarrow  \{ \text{Kernels of} \ \mathcal{S} \ \text{containing} \ L \}$$
given by $K/L \mapsto K$.
\end{thm}

\subsection{Principal kernels}$ $

Here are more properties of kernels  of \semifields0\ in terms of
their generators. $\mathcal{S}$ always denotes an idempotent
\semifield0.

\begin{defn}\label{defn_generation_of_kernels}
For a subset $S$ of   $\mathcal{S}$, denote by $\langle S \rangle$
the smallest
 kernel in $\mathcal{S}$ containing~$S$, i.e., the intersection of all kernels in $\mathcal{S}$
 containing $S$.
A kernel $K$ is said to be \textbf{finitely generated}  if $K =
\langle S \rangle$ where $S$ is a finite set.  If $K = \langle a
\rangle$ for some $a \in \mathcal{S}$, then $K$ is called a
\textbf{principal kernel}. \end{defn}

For convenience, we only consider kernels of polynomials with tangible coefficients; in \cite{Kern}
we treated the more general situation of arbitrary coefficients.
We say that $f \in F(\Lambda)$ is  \textbf{positive} if $f(\bfa)\ge 1$ for each $\bfa \in F^{(n)}$.
Given $f = \frac hg$ for $h,g \in \tT(\Lambda),$ we define $|f| = f + f^{-1}.$ Clearly $|f|$ is positive,
and $|f| = 1$ iff $f = 1.$

\begin{lem}\label{lem_generators}\cite[Property 2.3]{Prop_Semifields}
Let $K$ be a kernel  of an idempotent \semifield0 $\mathcal{S}$.
Then for $a,b \in \mathcal{S}$,
\begin{equation}
|a| \in K \ \ \text{or} \ \ |a|+ b \in K \ \ \Rightarrow a \in K.
\end{equation}
\end{lem}

\begin{prop}\label{prop_principal_ker}\cite[Proposition (3.1)]{Prop_Semifields}
\begin{equation}
\langle a \rangle = \{ x \in \mathcal{S} \ : \ \exists n \in
\mathbb{N} \  \ \ \text{such that} \ \ \  a^{-n} \leq x \leq a^{n}
\}.
\end{equation}
\end{prop}

\begin{cor}\label{cor_principal_ker_by_order}\cite[Corollary 4.1.16]{Kern}
For any $a \in \mathcal{S}$,
\begin{align*}
\langle a \rangle = \ & \{ x \in \mathcal{S} \ : \ \exists n \in
\mathbb{N} \ \text{such that} \  |a|^{-n} \leq   x \leq  |a|^{n} \}.
\end{align*}
\end{cor}

\begin{defn}\label{defn_generation_of_kernels1} A \semifield0 is said to be
\textbf{finitely generated} if it is finitely generated as a kernel.
If $\mathcal{S} = \langle a \rangle$ for some $a \in \mathcal{S}$,
then $\mathcal{S}$ is said to be a \textbf{principal \semifield0},
with \textbf{generator} $a$.
\end{defn}

%

\begin{thm}\label{thm_semi_with_a_gen}\cite[Theorem 4.1.19]{Kern}
If an archimedean idempotent \semifield0\ $F$ has a finite number of
generators $a_1, \dots, a_n$, then $F$ is a principal
\semifield0, generated by $a  = |a_1|   + \dots + |a_n| .$
\end{thm}

%
%

 \begin{note}\cite{Kern}  In view of Proposition~\ref{prop_principal_ker}, $\mathcal{S} =
 \langle \alpha \rangle$ for each $\alpha \ne \{1\}$.
\end{note}
%

\begin{defn} By \textbf{sublattice} of the lattice of kernels, we
mean a subset that is a lattice with respect to intersection and
multiplication. \end{defn}

\begin{cor}\label{cor_principal_kernels_sublattice}\cite[Corollary 4.1.26]{Kern}
The set of principal kernels of an idempotent \semifield0\ forms a
sublattice of the lattice of kernels.
\end{cor}

\begin{cor}\label{prop_frac_semifield_finitely_gen}\cite[Corollary 4.1.27]{Kern}
For any generator $a$ of a \semifield0 $F$,  $F(\Lambda) = \langle a
\rangle \prod_{i=1}^{n} \langle \la_i \rangle$,  and $F(\Lambda)$ is
a principal \semifield0\ with generator $\sum_{i=1}^n|\la_{i}|+
|a|$.
\end{cor}
.

\subsection{$\nu$-kernels}$ $

Let us make this all supertropical.

 \begin{defn}\label{vsemi24} A $\nu$-\textbf{congruence} on a \vdomain0 $R$ is a  congruence $\Cong$
 for which $(a,b) \in \Cong$ iff $(a^\nu, b^\nu) \in \Cong.$
We write $a_1 \equiv_\nu  a_2$ when  $a_1^\nu \equiv  a_2^\nu.$
\end{defn}

%
%

\begin{rem}\cite[Remark~4.1.27]{Kern}
For any congruence $\Cong$ of $\tG$,   $\nu^{-1}(\Cong) := \{(a,b):
a\nucong b\}$  is  a $\nu$-congruence of $R$.

Any $\nu$-congruence $\Cong = \{(a,b): a,b \in R\}$ of $R$ defines a
congruence $\Cong^\nu = \{(a^\nu ,b^\nu ): (a,b) \in \Cong \}$ of
$\tG$. Conversely, if $\Cong_\nu $ is a congruence of $\tG$, then
$\nu^{-1}(\Cong_\nu )$ is a $\nu$-congruence of $R$.

\end{rem}

Similarly, we have:

 \begin{defn}\label{vsemi29}
 A $\nu$-\textbf{kernel} of a \vsemifield0 $\mathbb S$ is a
subgroup $K$ which is $\nu$-\textbf{convex} in the sense that if
$a,b \in K$ and $\a, \beta \in \mathbb S$ with $\a + \beta \nucong
\fone,$ then $\a a + \beta b \in K.$

\end{defn}

\begin{rem}
 Any $\nu$-kernel ${\mathcal K}$ of a
$\nu$-\semifield0  $\mathbb S$ defines a kernel ${\mathcal K}^\nu$
of $\tG$. Conversely, if ${\mathcal K} $ is a kernel of $\tG$, then
$\nu^{-1}({\mathcal K})$ is a $\nu$-kernel of  $\mathbb S$.

 If
${\mathcal K} = {\mathcal K}_\Cong$, then $\nu^{-1}({\mathcal K}) =
{\mathcal K}_{\nu^{-1}(\Cong)}$.
\end{rem}

%
%
%
%

We recall the  monoid automorphism $(*)$ of order 2 of
\cite[Remark~3.3.1]{Kern} given by
$$a^* =   a^{-1}, \qquad (a^\nu)^*  = (a^{-1})^\nu, \qquad a \in \tT.$$

We also define the lattice supremum $a\wedge b = a+b$ and
\begin{equation}\label{starry}a\wedge b = (a^*+b^*)^*,\end{equation}
which is the lattice infinum.

\begin{rem}\label{bfacts1}\cite[Remark~4.2.5]{Kern}  $ $
\begin{enumerate} \eroman
\item  Given a $\nu$-congruence $\Cong$ on a $\nu$-\semifield0~$\mathbb S,$ we
define $K_\Cong = \left\{ a  \in \mathbb S  : a \equiv_\nu
1\right\}$. Conversely,  given a $\nu$-kernel $K$ of $\mathbb S$, we
define the $\nu$-congruence $\Cong $ on $\mathbb S$ by  $a \equiv b$
iff $ a b^* \equiv_\nu \fone.$

 \item   Any  $\nu$-kernel $K$ is $\nu$-convex,
  in the sense that if $a \le_\nu  b \le_\nu  c$ with $a,c \in K$,
then $b \in K.$

\item  If $|a|\in K$, a $\nu$-kernel, then $a\in K$.

 \item The product of two $\nu$-kernels is a $\nu$-kernel.

 \item The intersection of  $\nu$-kernels is a $\nu$-kernel.

\item   Any $\nu$-kernel generated by a finite set $\{ s_1, \dots,
s_m\}$ is generated by the single element $\sum _{i=1}^m (|s_i|).$

\item The $\nu$-kernel generated by $a \in \mathbb S$ is just the set of
finite sums $  \{ \sum_i  b_i  a^i : \ b_i \in \mathbb S, \ \sum b_i
\nucong 1\}.$

\item   If $K$ is a $\nu$-kernel of a \vsemifield0 $\mathbb S$ and the
\vsemifield0 $\mathbb S/K$ is $\nu$-idempotent, then $K$ is a
sub-\vsemifield0 of $\mathbb S$. (This is because for $a,b \in K$
the image of $a+b$ is $\nu$-equivalent to $1K + 1K = 1K.$)

\end{enumerate}
\end{rem}

In view of (v) for any set $S \subset \mathbb S$ we can define the
$\nu$-kernel $\langle S \rangle$  \textbf{generated by} $S$ to be
the intersection of all $\nu$-kernels containing $S$. {\it In what
follows, we only consider $\nu$-kernels generated by tangible
elements, in order to avoid ``ghost kernels'' and obtain the
following observation.}

\begin{prop}\label{prop_maximal_kernels_in_semifield_of_fractions_part1}
For any $\gamma_1,....,\gamma_n \in F$ the kernel $\langle
\frac{\la_1}{\gamma_1},...,\frac{\la_n}{\gamma_n} \rangle$ is a
maximal kernel of $F(\Lambda)$.
\end{prop}
\begin{proof} The quotient is isomorphic to $F$, which is simple.
\end{proof}

\subsection{The kernel $\langle F \rangle$}$ $

$\langle F \rangle$ denotes the kernel of $F(\Lambda)$  generated by
any element $\a \ne 1$ of $F$. This kernel plays a special role in
the theory, as seen in \cite{Kern}.

 \begin{cor}\label{gen5}
$F(\Lambda) =   F \cdot L_{(\alpha_1,...,\alpha_n)} = \langle F
\rangle \cdot L_{(\alpha_1,...,\alpha_n)}$.\end{cor}

\begin{lem}\label{prop_maximal_kernels_in_semifield_of_fractions_part2}
If $K$ is a maximal kernel of  $\langle F \rangle$, then
$$K \in \Omega\left(\left\langle \frac{\la_1}{\alpha_1}, ... ,
\frac{\la_n}{\alpha_n} \right\rangle\right)$$ for suitable
$\alpha_1,...,\alpha_n \in F$.
\end{lem}
\begin{proof}
Denote $L_a = (|\frac{\la_1}{\alpha_1}| + ....+
|\frac{\la_n}{\alpha_n}|)    \wedge |\alpha|$ with $\alpha \neq
1$, for $a = (\alpha_1, ..., \alpha_n)$. We may assume that
$\Skel(K) \neq \emptyset$, since the kernel corresponding to the
empty set is $\langle F \rangle$ itself. If $a \in \Skel(K)$, then
$\langle L_a \rangle \supseteq K$ since $\Skel(L_a) = \{a\}
\subseteq \Skel(K)$. Thus, the maximality of $K$ implies that $K =
\langle L_a \rangle$.
\end{proof}

\subsection{The Zariski correspondence for $\nu$-kernels}\label{Zarcor}$ $

\begin{defn}
A \textbf{kernel root} of $ f  \in \FunF$  is an element $\bfa \in
F^{(n)}$ such that $ f (\bfa) \nucong \fone.$

For $S \subseteq  F(\Lambda)$, define
\begin{equation}
\Skel(S) = \{ \bfa \in F^{(n)}  \ : \  f(\bfa) \nucong 1, \ \forall
f \in S \}.
\end{equation}
\end{defn}

We write  $\Skel(f)$ for $\Skel(\{f\}).$

%
%
%

\begin{defn}
A subset $Z\subset F^{(n)}$ is said to be a \textbf{$\onenu$-set} if
there exists a subset $S \subset F(\Lambda)$  such that $Z =
\Skel(S)$.
\end{defn}

\subsection{The coordinate \vsemifield0 of a $\onenu$-set}$ $

\begin{defn}\label{def_coord_semifield_1} For $X \subset F^{(n)},$ The \textbf{coordinate
\vsemifield0} $F(X)$ of a $\onenu$-set $X$ is the set of restriction
of the rational functions $F(\Lambda)$ to $X$.

$$\phi_{X}~:~  F(\Lambda)~\rightarrow~F(X)$$
denotes the restriction map  $h \mapsto h|_{X}$.
\end{defn}

The tropical significance comes from:

\begin{thm}\label{cor11}\cite[Theorem~7.1.7]{Kern} The correspondences $f \mapsto \hat f$ and $h \mapsto
\underline{h}$ induces a 1:1 correspondence between corner
hypersurfaces and $\onenu$-sets of corner internal rational
functions.
\end{thm}

\begin{prop}\label{prop_coord_semifield_100}\cite[Proposition~5.4.2]{Kern}
$\phi_{X}$ is an onto \semifield0 homomorphism.
\end{prop}

\begin{prop}\label{prop_coord_semifield_1}\cite[Proposition~5.4.3]{Kern}
$F(X)$ is a \vdomain0, isomorphic to $F(\Lambda)/\Ker(X).$
\end{prop}

Thus, chains of kernels of $F(\Lambda)$ give us an algebraic
tropical notion of dimension, and its determination is the subject
of this paper.

\subsection{The Jordan–-H\"{o}lder theorem}$ $

Our main goal is to find a Jordan–-H\"{o}lder theorem for kernels.
But there are too many kernels for a viable theory in general, as
discussed in \cite{Bi}. If we limit our set of kernels to a
sublattice of kernels, one can use the Schreier refinement theorem
\cite{Rot} to obtain a version of the Jordan--H\"{o}lder Theorem.

\begin{defn}\label{natural}
 $\mathcal L(\mathcal S)$ denotes the lattice
of $\nu$-kernels of a \vsemifield0 $\mathcal S.$

$\Theta$ is a \textbf{natural map} if for each \vsemifield0
$\mathcal S,$ there is a lattice homomorphism $\Theta_{\mathcal S}:
\mathcal L(\mathcal S)\to \mathcal L(\mathcal S) $ such that $K
\mapsto \Theta_{\mathcal S}(K)$ is a homomorphism of kernels. We
write $\Theta(\mathcal S) $ for $\Theta_{\mathcal S}(\mathcal
L(S)),$ and call the kernels in $\Theta_{\mathcal S}$
$\Theta$-\textbf{kernels}. (We delete $\mathcal S$ when it is
unambiguous.)

 A $\Theta(\mathcal S)$-\textbf{simple} kernel is a minimal
$\Theta$-kernel $\ne \{1 \}$.  A $\Theta(\mathcal
S)$-\textbf{composition series}
 $\mathcal C(K,L)$ in $\Theta(\mathcal S)$ from  a kernel $K$ to a subkernel $L$ is
  a chain $$K = K_0 \supset K_1 \supset \dots K_t = L$$  in $\Theta$
such that each factor is $\Theta$-simple. \end{defn}

By Theorem~\ref{latticecor} $\mathcal C(K,L)$ is equivalent to the
$\Theta(\mathcal S/L)$-composition series
$$K/L \supset K_1/L \supset \dots \supset K_t/L = 0$$ of $K/L.$

Given a $\Theta$-kernel $K$, we define its {\bf composition length}
$\ell (K)$ to be the length of a $\Theta$-composition series for $K$
(presuming $K$ has one). By definition, $\{ 1\}$ is the only
$\Theta$-kernel of composition length 0. A nonzero $\Theta$-kernel
$K$ is simple iff $\ell(K) = 1.$ The next theorem is a standard
lattice-theoretic result of Schreier and Zassenhaus, yielding the
Jordan-H\"{o}lder Theorem, cf.~\cite [Theorem~3.11,
Schreier-Jordan-H\"{o}lder Theorem]{Ro}.

\begin{thm}\label{Schreier}  Suppose $K$ has a
composition series  $$K = K_0 \supset K_1 \supset \dots \supset K_t
= 0,$$  which we denote as $\mathcal C$. Then:

(i) Any arbitrary finite chain of subkernels $$K = N_0 \supset N_1
\supset \dots \supset N_k\supset 0$$ (denoted as $\mathcal D$),  can
be refined to a composition series equivalent to $\mathcal C$. In
particular, $k\le t$.

(ii)  Any two composition series of $K$ are equivalent.

(iii) $\ell (K) = \ell(N) + \ell (K/N)$ for every subkernel $N$ of
$K$. In particular, every subkernel and every homomorphic image of a
kernel with composition series has a composition series.
\end{thm}

\subsection{The HO-decomposition}\label{HOdecomp}$ $

\begin{defn} For any kernel $K$ of $F(\Lambda)$, define
the equivalence relation

\begin{equation}
f \sim_K f'\quad \text{if and only if} \quad \langle f \rangle
\cap K = \langle f'\rangle  \cap K
\end{equation}as kernels of $F(\Lambda)$.   The equivalence classes   are
$$[f] = \{ f' \ : \ f' \ \text{is a generator of} \ \langle f
\rangle \cap  K \}.$$ 
\end{defn}

Our interest is in $K = {\FF}.$
\begin{defn}
An $\scrL$-\textbf{monomial} is a non-constant Laurent monomial $f
\in F(\la_1,...,\la_n)$; i.e.,  $f = \frac{h}{g}$ with $h,g \in
F[\la_1,...,\la_n]$ non-proportional monomials.

%
%

A rational function  $f \in F(\la_1,...,\la_n)$ is called a
\textbf{hyperspace-fraction}, or HS-fraction, if $f \sim_{\FF}
\sum_{i=1}^{t} |f_i|$ where  the  $f_i $ are non-proportional
$\scrL$-monomials.
\end{defn}

\begin{defn}
A $\onenu$-set  in $F^{(n)}$ is  a \textbf{hyperplane $\onenu$-set
} (HP-$\onenu$-set for short) if it is defined by an
$\scrL$-monomial. A $\onenu$-set  in $F^{(n)}$ is   a
\textbf{hyperspace-fraction $\onenu$-set } (HS-$\onenu$-set for
short) if it is defined by an HS-fraction.
\end{defn}

\begin{prop} \cite[Corollary~9.1.10]{Kern}
A $\onenu$-set  is an HS-$\onenu$-set  if and only if it is an
intersection of HP-$\onenu$-sets.
\end{prop}

%

 $\PCon(K )$ denotes the lattice
of principal subkernels  of a kernel $K.$

\begin{defn}
$ \operatorname{HP}(K)$  denotes the family of $\scrL$-monomials in
a kernel $K$.
 A \textbf{hyperplane kernel}, or  \textbf{HP-kernel},
for short, is a principal kernel of $F(\la_1,...,\la_n)$ generated
by an $\scrL$-monomial.

A \textbf{hyperspace-fraction kernel}, or  \textbf{HS-kernel}, for
short, is a principal kernel of $F(\la_1,...,\la_n)$ generated by
a hyperspace fraction.

\end{defn}

\begin{defn}\label{Omega}
 $\Omega(F(\Lambda))$ is the lattice of
kernels finitely generated by  HP-kernels of $F(\Lambda)$, i.e.,
every element $\langle f \rangle \in \Omega(F(\Lambda))$ is obtained
via finite intersections and products of HP-kernels.
\end{defn}

\begin{prop}\label{remHSprodHP}\cite[Proposition (9.1.7)]{Kern}
Any principal HS-kernel is a product of distinct HP-kernels, and
thus is in $\Omega(F(\Lambda))$.
\end{prop}

\begin{lem}\label{prop_HP_element_is_a_generator}\cite[Lema~9.1.11]{Kern}
Let $\langle f \rangle$ be an HP-kernel, with $F$ divisible. If $w
\in \langle f \rangle$ is an $\scrL$-monomial, then $w^s = f^k$ for
some $s,k \in \mathbb{Z} \setminus\{ 0 \}$.
\end{lem}

Before refining this description, we recall \cite[Construction~2.6.1]{Kern}, to fix notation.

\begin{constr}\label{HO_construction}
Take a rational function  $f \in F(\la_1,...,\la_n)$ for which
$\Skel(f)\ne \emptyset$. Replacing $f$ by $|f|$, we may assume that
$f \geq_\nu 1$. Write $f = \frac{h}{g} = \frac{\sum_{i=1}^k
h_i}{\sum_{j=1}^m g_j}$ where $h_i$ and $g_j$ are monomials in
$F[\la_1,...,\la_n]$. For each $\bfa\in \Skel(f)$, let
$$H_a \subseteq H = \{ h_i \ : \ 1 \leq i \leq k \}; \qquad G_a \subseteq G = \{ g_j \ : \ 1 \leq j \leq m \}$$
be the sets of dominant monomials   at $\bfa$; thus,  $ h_i(\bfa) =
g_j(\bfa) $ for any $h_i \in H_\bfa$ and $g_j \in G_\bfa$. Let
$H_{\bfa}^{c} = H \setminus H_\bfa$ and $G_{\bfa}^{c} = G \setminus
G_\bfa$. Then, for any $h' \in H_\bfa$ and $h'' \in H_{\bfa}^{c}$, $
h'(\bfa)  + h''(\bfa) = h'(\bfa),$ or, equivalently,    $1 +
\frac{h''(\bfa)}{h'(\bfa)} = 1$. Similarly, for any $g' \in G_\bfa$
and $g'' \in G_{\bfa}^{c}$, $g'(\bfa)  + g''(\bfa) = g'(\bfa)$  or,
equivalently,  $1 + \frac{g''(\bfa)}{g'(\bfa)} = 1$.

Thus for any such $\bfa$ we obtain the  relations
\begin{equation}\label{eq_relations1}
\frac{h'}{g'} = 1, \qquad \forall h' \in  H_{\bfa},\ g' \in
G_{\bfa},
\end{equation}
\begin{equation}\label{eq_relations2}
1  +   \frac{h''}{h'} = 1  ; \ 1  +   \frac{g''}{g'} = 1,  \qquad
\forall h' \in H_{\bfa},\ h'' \in H_{\bfa}^{c},\ g' \in G_{\bfa},\
g'' \in G_{\bfa}^{c}.
\end{equation}
As $\bfa $ runs over $\Skel(f)$,  there are only finitely many
possibilities for $H_a$ and $G_a$ and thus for the relations in
\eqref{eq_relations1} and~\eqref{eq_relations2}; we denote these as
$(\theta_1(i), \theta_2(i)),$ $i = 1,...,q$.

In other words, for any $1 \le i \le q,$ the pair $(\theta_1(i),
\theta_2(i))$ corresponds to a kernel $K_i$ generated by the
corresponding elements
$$\frac{h'}{g'}  ,\ \left(1  +   \frac{h''}{h'}\right),
\ \text{and} \  \left(1  +   \frac{g''}{g'}\right),$$ where $\{
\frac{h'}{g'} = 1 \}  \in \theta_1$ and $\{ 1  +   \frac{g''}{g'} =
1\},\ \{ 1  +   \frac{h''}{h'} = 1\} \in \theta_2$.

Reversing the argument, every point satisfying one of these $q$ sets
of relations  is in $\Skel(f)$. Hence,
\begin{equation}\label{localexpansion} \Skel\left(\langle f \rangle
\cap \langle F \rangle \right) =\Skel(f) =
\bigcup_{i=1}^{q}\Skel(K_i) = \bigcup_{i=1}^{q}\Skel\left(K_i \cap
\langle F \rangle\right)\end{equation} $$=
\Skel\left(\bigcap_{i=1}^{q}(K_i \cap \langle F \rangle)\right),$$
Hence  $\langle f \rangle \cap \langle F \rangle =
\bigcap_{i=1}^{q}K_i \cap \langle F \rangle$, since $\langle f
\rangle \cap \langle F \rangle, \ \bigcap_{i=1}^{q}K_i \cap \langle
F \rangle \in \PCon(\langle F \rangle)$.
 $\bigcap_{i=1}^{q}K_i$ provides a local description of $f$ in a neighborhood of its $\onenu$-set .

Let us view this construction globally. We used the $\onenu$-set  of
$\langle f \rangle$ to construct $\bigcap_{i=1}^{q}K_i$. Adjoining
various points $\bfa$ in $F^{(n)}$ might add some regions,
complementary to the regions defined by \eqref{eq_relations2} in
$\theta_2(i)$ for $i=1,...,q$, over which $\frac{h'}{g'} \neq 1,
\forall h' \in H_{\bfa}, \forall g' \in G_{\bfa}$ for each $\bfa$,
i.e., regions over which the dominating monomials never agree.
Continuing the construction above using  $\bfa \in F^{(n)} \setminus
\Skel{(f)}$ similarly produces a finite collection of, say $t \in
\mathbb{Z}_{\geq 0}$, kernels generated by elements from
\eqref{eq_relations2} and their complementary order fractions and by
elements of the form \eqref{eq_relations1} (where now $\frac{h'}{g'}
\neq 1$ over the region considered). Any
 principal kernel $N_j = \langle q_j \rangle$,
   $1 \leq j \leq t$, of this complementary set of kernels has the
   property that $\Skel(N_j) = \emptyset$, and thus by
   Corollary \ref{cor_empty_kernels_correspond_to_bfb_kernels}, $N_j$ is bounded from below. As there are finitely many such kernels there exists small enough $\gamma >_\nu 1$ in $\tT$
   for which $|q_j| \wedge \gamma = \gamma$ for $j = 1,...,t$. Thus $\bigcap_{j=1}^{t}N_j$ is bounded from below and thus  $\bigcap_{j=1}^{t}N_j \supseteq \langle F \rangle$ by Remark
   \ref{rem_bounded_from_below_contain_H_kernel}.

Piecing this together with \eqref{localexpansion} yields $f$ over
all of $F^{(n)}$, so we have
\begin{equation}\label{full_expansion}
\langle f \rangle = \bigcap_{i=1}^{q}K_i \cap \bigcap_{j=1}^{t}N_j.
\end{equation}

So, $\langle f \rangle \cap \langle F \rangle = \bigcap_{i=1}^{q}K_i \cap \bigcap_{j=1}^{t}N_j \cap \langle F \rangle = \bigcap_{i=1}^{q}K_i \cap \langle F \rangle$.\\
\end{constr}

In  this way, we see that intersecting a principal
 kernel $\langle f \rangle$  with $\langle F \rangle$ `chops off' all of the  bounded from below kernels in \eqref{full_expansion} (the $N_j$'s given above). This eliminates ambiguity in the kernel corresponding to $\Skel(f)$.\\
Finally we note that if $\Skel(f) = \emptyset$, then $\langle f \rangle = \bigcap_{j=1}^{t}N_j$ for appropriate kernels $N_j$ and   $\langle f \rangle \cap \langle F \rangle = \langle F \rangle$.\\

\begin{rem}$ $
\begin{enumerate}\eroman
  \item If $K_1$ and $K_2$ are such that $K_1  K_2 \cap F = \{1\}$ (i.e., $\Skel(K_1) \cap \Skel(K_2) \neq \emptyset$), then the sets of $\scrL$-monomials~$\theta_1$ of $K_1$ and of $K_2$ are not the same
   (although one may contain the other), for otherwise together they would yield a single kernel via Construction~\ref{HO_construction}.

  \item The kernels $K_i$, being finitely generated, are in fact principal,
  so we can write $K_i = \langle k_i \rangle$ for rational functions $k_1,
\dots, k_q$. Let $\langle f \rangle \cap \langle F \rangle =
\bigcap_{i=1}^{q}(K_i \cap \langle F \rangle)  =
\bigcap_{i=1}^{q}\langle |k_i| \wedge |\alpha| \rangle =
\bigwedge_{i=1}^{q} \langle |k_i| \wedge |\alpha| \rangle$ with
$\alpha \in F \setminus \{ 1 \}$. By~\cite[Theorem~8.5.3]{Kern},
for any generator $f'$ of $\langle f \rangle \cap \langle F
\rangle$ we have  $|f'| = \bigwedge_{i=1}^{q} |k_i'|$ with $k_i'
\sim_{\FF} |k_i| \wedge |\alpha|$ for every $i = 1,...,q$. In
particular, $\Skel(k_i') = \Skel(|k_i| \wedge |\alpha|) =
\Skel(k_i)$. Thus the kernels $K_i$ are independent of the choice
of generator $f$, being defined by the components $\Skel(k_i)$ of
$\Skel(f)$.
\end{enumerate}
\end{rem}

Two instances of Construction~\ref{HO_construction} are given in
\cite[Examples~2.6.3, 2.6.4]{Kern}.

\begin{defn}
A rational function
$g \in F(\Lambda)$ is  \textbf{bounded from below} if
there exists some $\alpha >_\nu 1$ in $F$ such that $|g| \ge_\nu
\alpha$.

An important instance: the  \textbf{$\scrL$-binomial} $o$ defined by an $\scrL$-monomial
$f$ is the rational function   $1 +   f$.
The \textbf{complementary $\scrL$-binomial} $o^{c}$ of $o$ is $ 1 +
f^{-1}$. By definition $(\mathcal O^c)^c = \mathcal O$.
The \textbf{order kernel} of the \semifield0 $F(\la_1,...,\la_n)$
defined by $f$ is the principal kernel   $\mathcal O = \langle o
\rangle$ for the $\scrL$-binomial $o = 1  +   f$.
 The \textbf{complementary order kernel} $\mathcal O^c$  of $\mathcal O$ is $ \langle o^c \rangle$.

A rational function  $f \in F(\la_1,...,\la_n)$ is said to be a
\textbf{region fraction}   if   $\Skel (f)$ contains some nonempty
open interval. A \textbf{region kernel} is a principal kernel  generated by a
region fraction.

\end{defn}

\begin{lem}\cite[Lemma~9.1.16]{Kern}
 $f \sim_{\FF} \sum_{i=1}^{t} |o_i|$ is a region fraction iff,
 writing $o_i = 1+f_i$ for $\scrL$-monomials $f_i$, we have $f_i \not\nucong f_j^{\pm1}$ for every $i\ne
 j.$\end{lem}

\begin{defn}\label{defn_HO_fraction}
A rational function  $f \in F(\la_1,...,\la_n)$ is  an
\textbf{HO-fraction} if it is the sum of   an HS-fraction $f'$ and a
region fraction $o_f$. (In particular,   any HS-kernel or any
region kernel is an HO-kernel.)

A principal kernel $K \in \PCon(F(\la_1,...,\la_n))$ is said to be
an \textbf{HO-kernel} if it is generated by an HO-fraction.
\end{defn}

\begin{lem}\cite[Lemma~2.6.7]{Kern}
 A principal kernel $K$ is an HO-kernel if and only if
   $K = L  R$ where $L$ is an HS-kernel and $R$ is a region kernel.
\end{lem}

\begin{thm}\label{thm_HP_expansion}\cite[Theorem~2.6.8]{Kern}
Every principal kernel $\langle f \rangle$ of $F(\la_1,...,\la_n)$
can be written as the intersection of finitely many principal
kernels
$$\{K_i : i=1,...,q\} \ \text{and} \ \{N_j : j = 1,...,m \},$$
whereas each $K_i$ is the product of an HS-kernel and a region
kernel
\begin{equation}\label{consab}
K_i = L_i  R_i =  \prod_{j=1}^{t_i}L_{i,j} \prod_{k=1}^{k_i}\mathcal
o_{i,k}
\end{equation}
while each $N_j$ is a product of bounded from below kernels and
(complementary) region kernels. For $\langle f \rangle \in
\PCon(\langle F \rangle)$, the~$N_j$ can be replaced by $\langle F
\rangle$ without affecting $\langle f \rangle $.\end{thm}

\section{Convexity degree and
hyperdimension}\label{section:Convexity degree and Hyperdimension}

Let $\langle f \rangle \subseteq \langle F \rangle$ be a principal kernel
and let $\langle f \rangle = \bigcap_{i=1}^{s} K_i$, where
$$K_i = (L_i \cdot R_i) \cap \langle F \rangle = (L_i \cap \langle F \rangle) \cdot (R_i \cap \langle F \rangle) = L_i' \cdot R_i'$$
is its (full) HO-decomposition; i.e., for each $1 \leq i \leq s$, \
$R_i \in \PCon(F)$ is a region kernel and $L_i\in
\PCon(F)$ is either an HS-kernel or bounded from below
(in which case  $L_i' = \langle F \rangle$). Then by
Corollary~\ref{cor_subdirect_decomposition_for_semifield_of_fractions},
we have the subdirect decomposition
$$\langle F \rangle/ \langle f \rangle  \hookrightarrow \prod_{i=1}^{t} \langle F \rangle /K_i = \prod_{i=1}^{t} (\langle F \rangle /L_i' \cdot R_i')$$
where $t \leq s$ is the number of kernels $K_i$ for which $L_i' \neq
\langle F \rangle$ (for otherwise $\langle F \rangle /K_i~=~\{1 \}$
and can be omitted from the subdirect product).

\begin{exmp}\label{exmp_infinite_chain}
Consider the principal kernel $\langle \la_1 \rangle \in
\PCon(F(\la_1,\la_2))$. For $\alpha \in F$ such that $\alpha > 1$,
we have the following infinite strictly descending chain of
principal kernels
$$\langle \la_1 \rangle \supset \langle |\la_1| + |\la_2+1| \rangle \supset \langle |\la_1| + |\alpha^{-1}\la_2 + 1| \rangle \supset \langle |\la_1| + |\alpha^{-2}\la_2 + 1| \rangle \supset \dots$$ $$\supset \langle |\la_1| + |\alpha^{-k}\la_2 + 1| \rangle \supset \dots$$ and the strictly ascending chain of $\onenu$-sets corresponding to it.
$$\skel(\la_1) \subset \skel(|\la_1| + |\la_2+1|) \subset \dots \subset \skel(|\la_1| + |\alpha^{-k}\la_2 + 1|) \subset \dots =$$ $$\skel(\la_1) \subset \skel(\la_1) \cap \skel(\la_2+1) \subset \dots \subset \skel(\la_1) \cap \skel(\alpha^{-k}\la_2 + 1) \subset \dots.$$
\end{exmp}
%
%
%
%
%
%
%
%
%

\begin{exmp}\label{exmp_decomposition}
Again, consider the principal kernel $\langle x \rangle \in
\PCon(F(x,y))$. Then $$\langle x \rangle = \langle |x| + (|y + 1|
\wedge |\frac{1}{y} + 1|) \rangle = \langle  (|x| + |y + 1|)
\wedge (|x| + |\frac{1}{y} + 1|) \rangle = \langle |x| + |y + 1|
\rangle \cap \langle |x| + |\frac{1}{y} + 1| \rangle.$$ So, we
have the nontrivial decomposition of $\skel(x)$ as $\skel(|x| + |y
+ 1|) \cup \skel(|x| + |\frac{1}{y} + 1|)$ (note that $\skel(|x| +
|y+1|) = \skel(x) \cap \skel(y+1)$,  and furthermore $\skel(|x| +
|\frac{1}{y} + 1|) = \skel(x) \cap \skel(\frac{1}{y} + 1)$). In a
similar way, using complementary order kernels,  one can show that
every principal kernel can be nontrivially decomposed to a pair of
principal kernels.
\end{exmp}
%
%
%
%
%
%
%
%
%

\subsection{Reducible kernels}$ $

Examples \ref{exmp_infinite_chain} and \ref{exmp_decomposition}
demonstrate that the lattice of principal kernels
$\PCon(F(\Lambda))$ (resp. $\PCon(\langle F \rangle)$) is too rich
to define reducibility or finite dimension. (See \cite{Bi} for a
discussion of infinite dimension.) Moreover, these examples suggest
that this richness is caused by order kernels. This motivates us to
consider $\Theta$-reducibility for a suitable sublattice of kernels
$\Theta \subset \PCon(F(\Lambda))$ (resp. $\Theta \subset
\PCon(\langle F \rangle)$).

There are various  families of
 kernels that could be utilized to define the notions of reducibility,
dimensionality, and so forth. We take $\Theta$ to be the
sublattice generated by HP-kernels, because of its connection to
the (local) dimension of the linear spaces (in logarithmic scale)
defined by the $\onenu$-set corresponding to a kernel. Namely,
HP-kernels, and more generally   HS-kernels, define affine
subspaces of $F^{(n)}$ (see \cite[\S 9.2]{Kern}).
%
%
%
 We work with Definition~\ref{Omega}.

\begin{defn}\label{irre} A kernel $\langle f \rangle \in \Omega(F(\Lambda))$ is
\textbf{reducible} if there are  $ \langle g \rangle, \langle h
\rangle \in \Omega(F(\Lambda))$ for which  $ \langle g \rangle,
\langle h \rangle \not \subseteq \langle f \rangle$ but $\langle g
\rangle \cap \langle h \rangle \subseteq \langle f \rangle$.
\end{defn}


\begin{lem}\label{irred}
 $\langle f \rangle$ is
reducible iff  $\langle f \rangle  = \langle g \rangle \cap \langle
h \rangle$ where $\langle f \rangle \neq \langle g \rangle$ and
$\langle f \rangle \neq \langle h \rangle.$
\end{lem}
\begin{proof}
Assume $\langle f \rangle$ admits the stated condition.
If $\langle f \rangle  \supseteq \langle g \rangle \cap \langle h \rangle$, then $\langle f \rangle =
\langle f \rangle \cdot \langle f \rangle  = (\langle g \rangle \cdot \langle f \rangle )\cap (\langle h \rangle \cdot \langle f \rangle)$. Thus $\langle f \rangle = \langle g \rangle \cdot \langle f \rangle$ or $\langle f \rangle = \langle h \rangle \cdot \langle f \rangle$, implying $\langle f \rangle \supseteq \langle g \rangle$ or $\langle f \rangle \supseteq \langle g \rangle$. The converse is obvious.\end{proof}

\begin{lem}
Let $\langle f \rangle$ be an HP-kernel. Then for any HP-kernels
$\langle g \rangle$ and $\langle h \rangle$ such that $\langle f
\rangle = \langle g \rangle \cap \langle h \rangle$ either
$\langle f \rangle = \langle g \rangle$ or $\langle f \rangle =
\langle h \rangle$. In other words, every HP-kernel is
irreducible.
\end{lem}
\begin{proof}
If $\langle f \rangle = \langle g \rangle \cap \langle h \rangle$
then $\langle f \rangle \subseteq  \langle g \rangle$ thus $f \in
\langle g \rangle$. As both  $f$ and $g$ are $\scrL$-monomials (up
to equivalence),   Lemma~\ref{prop_HP_element_is_a_generator} yields
$\langle f \rangle = \langle g \rangle$, which in turn, by
Lemma~\ref{irred}, implies that $\langle f \rangle$ is irreducible.
\end{proof}
\begin{cor}\label{cor_HS-decomposition}
Any HS-kernel $\langle f \rangle$ is irreducible.
\end{cor}
\begin{proof}
If $\langle f \rangle =  \langle g \rangle \cap \langle h \rangle$
for HP-kernels $\langle g \rangle$ and $\langle h \rangle$, then
$\langle f \rangle \subseteq \langle g \rangle$. But $\langle f
\rangle$ is a product $\langle f_1 \rangle \cdots \langle f_t
\rangle$  of finitely many HP-kernels. For each $1 \le j \le t$,
$\langle f_j \rangle \subseteq \langle g \rangle$  yielding $\langle
g \rangle = \langle f_j  \rangle$  by
Lemma~\ref{prop_HP_element_is_a_generator}, and so $\langle f
\rangle = \langle g \rangle$.
\end{proof}

\begin{cor}\label{cor_reducibility_of_HS_kernels}
If $\langle f \rangle = \langle g \rangle \cap \langle h \rangle$
for  HS-kernels $\langle g \rangle$ and $\langle h \rangle$, then
either $\langle f \rangle = \langle g \rangle$ or $\langle f \rangle
= \langle h \rangle$.
\end{cor}
\begin{proof} Otherwise, since
$\langle g \rangle$ and $\langle h \rangle$ are finite products of
HP-kernels, there are HP-kernel $\langle g' \rangle \subseteq
\langle g \rangle$   and  $\langle h' \rangle \subseteq \langle h
\rangle$   such that $\langle g' \rangle \not \subseteq \langle f
\rangle$ and $\langle h' \rangle \not \subseteq \langle f \rangle$.
But $\langle g' \rangle \cap \langle h' \rangle \subseteq \langle g
\rangle \cap \langle h \rangle = \langle f \rangle$, implying
$\langle f \rangle$ is reducible, contradicting Corollary
\ref{cor_HS-decomposition}.
\end{proof}

\begin{prop} The irreducible kernels in the
  lattice generated by HP-kernels are precisely the HS-kernels.\end{prop}
\begin{proof}
 This follows from Corollary \ref{cor_reducibility_of_HS_kernels}, since all proper intersections in the lattice generated by HP-kernels are reducible.
  (Note that HP-kernels are also HS-kernels.)
\end{proof}

\begin{cor}
$\HSpec(F(\Lambda))$ is the family of HS-kernels in
$\Omega(F(\Lambda))$, which is precisely the family of HS-fractions of
$F(\Lambda)$.
\end{cor}

\begin{defn}\label{hyperspec}
The \textbf{hyperspace spectrum} of $F(\Lambda)$, denoted
$\HSpec(F(\Lambda))$, is the family of irreducible kernels in
$\Omega(F(\Lambda))$.
\end{defn}
%

\begin{defn}\label{height0}
A chain $P_0 \subset P_1 \subset \dots \subset P_t$ in
$\HSpec(F(\Lambda))$  of HS-kernels of $F(\Lambda)$ is said to have
\textbf{length}~$t$. An HS-kernel $P$ has \emph{height} $t$ (denoted
$\hgt(P)=t$) if there is a chain of length $t$ in
$\HSpec(F(\Lambda))$ terminating at $P$, but no chain of length
$t+1$ terminates at $P$.
\end{defn}

\begin{rem}\label{rem_correspondence_of_quontient_HSpec}
Let $L$ be a kernel in $\PCon(F(\Lambda))$. Consider the canonical
homomorphism $\phi_L : F(\Lambda) \rightarrow F(\Lambda)/L$. Since
the image of a principal kernel is generated by the image of any of
its generators, $\phi_L(\langle f \rangle)~=~\langle \phi_L(f)
\rangle$ for any HP-kernel $\langle f \rangle$. Choosing $f$ to be
an $\scrL$-monomial,  $\langle \phi_L(f) \rangle$ is a nontrivial
HP-kernel in $ F(\Lambda)/L$ if and only if $\phi_L(f) \not \in F$.
Thus, the set of HP-kernels of $F(\Lambda)$ mapped to HP-kernels of
$F(\Lambda)/L$ is
\begin{equation}\label{eq_subset_of_Omega}
\left\{ \langle g \rangle : \langle g \rangle \cdot \langle F \rangle \supseteq \phi_L^{-1}\left(\langle F \rangle\right) = L \cdot \langle F \rangle \right\}.
\end{equation}
As $\phi_L$ is an   $F$-homomorphism, it respects $\vee, \wedge$ and
$| \cdot |$, and  thus $\phi_{L}( (\Omega (F(\Lambda)), \cap,
\cdot)) = (\Omega (F(\Lambda)/L),\cap, \cdot)$.
 In fact Theorem~\ref{latticecor} yields a correspondence
identifying $\HSpec(F(\Lambda)/L)$ with the subset of
$\HSpec(F(\Lambda))$ which consists of all HS-kernels $P$ of
$F(\Lambda)$ such that   $P \cdot \langle F \rangle \supseteq
L\cdot\langle F \rangle$.  \end{rem}

\begin{lem} The above correspondence extends to a
correspondence identifying $\Omega(F(\Lambda)/L)$ with the
subset~\eqref{eq_subset_of_Omega} of $\Omega(F(\Lambda))$.  Under
this correspondence, the maximal HS-kernels of $F(\Lambda)/L$
correspond to maximal HS-kernels of $F(\Lambda)$, and reducible
kernels of $F(\Lambda)/L$ correspond to reducible kernels of
$F(\Lambda)$. \end{lem}
\begin{proof} The latter assertion is obvious since $\wedge$
is preserved under homomorphisms. For the first assertion,
$(F(\Lambda)/L)/(P/L) \cong F(\Lambda)/P$ by
Theorem~\ref{thm_nother_1_and_3}, so simplicity of the quotients is
preserved. Hence, so is maximality of  $P/L$ and $P$.
\end{proof}

\begin{defn}\label{defn_Hdim}
The Hyperdimension of $F(\Lambda)$, written $\Hdim F(\Lambda)$ (if
it exists), is the maximal height of the HS-kernels in $F(\Lambda)$.
\end{defn}

\subsection{Decompositions}$ $

Let us garner some information about reducible kernels from
rational functions. Suppose $f\in F(\Lambda)$. We write $f =
\sum_{i=1}^{k}f_i$ where each $f_i$ is of the form $ {g_i}{h_i}^*$
with $g_i,h_i \in F[\Lambda]$ and $g_i$  a monomial. (This is
$\frac{g_i}{h_i}$ when $h_i$ is tangible.) We also assume that
this sum is \textbf{irredundant} in the sense that we cannot
remove any of the summands and still get $f$.
If each time the value $1$ is attained by one of the terms $f_i$ in this expansion and all other terms attain values   $\le 1$, then $\tilde{f} = \bigwedge_{i=1}^{k}|f_i|$  defines the same $\onenu$-set as $f$. Moreover, if $f \in \langle F \rangle$ then $\tilde{f} \wedge |\alpha| \in \langle F \rangle$, for $\alpha \in F \setminus \{1\}$  is also a generator of $\langle f \rangle$. The reason we take $\tilde{f} \wedge |\alpha|$ is that we have no guarantee that each of the $f_i$'s in the above expansion is bounded.\\
We can generalize this idea as follows:

 We call $f\in F(\Lambda)$
\textbf{reducible} if we can write $f = \sum_{i=1}^{k}f_i$ as
above, such that for every $1 \leq i \leq k$ the following
condition holds:
\begin{equation*}
f_i(\bfa) \nucong 1  \Rightarrow  \ \ f_j(\bfa) \le _{\nu} 1, \
\forall j \neq i.
\end{equation*}

\begin{defn}\label{defn_decomposition}
Let $f \in F(\Lambda)$. A $\Theta$-\textbf{decomposition} of $f$
is an expression of the form
\begin{equation}\label{eq_decomposition}
|f| = |u| \wedge |v|
\end{equation}
with $u,v $ $\Theta$-elements in $F(\Lambda)$.

The decomposition \eqref{eq_decomposition} is said to be
\emph\textbf{trivial} if  \ $f \sim_{\FF} u$ \ or \ $f \sim_{\FF}
v$ (equivalently $|f|~\sim_{\FF}~|u|$ or $|f|~\sim_{\FF}~|v|$).
 Otherwise, the decomposition  is said to be \textbf{nontrivial}.
\end{defn}
%
%
%

\begin{lem}\label{lem_reducibility_and_decomposition}
Suppose $f \in F(\Lambda)$ is a $\Theta$-element. Then $\langle f
\rangle$ is reducible if and only if there exists some generator
$f'$ of $\langle f \rangle$  that has a nontrivial
$\Theta$-decomposition.
\end{lem}
\begin{proof}
If $\langle f \rangle$ is reducible , then there exist
 kernels $\langle u \rangle$ and $\langle v \rangle$ in
$\Theta$ such that $\langle f \rangle = \langle u \rangle \cap
\langle v \rangle$ where $\langle f \rangle \neq \langle u
\rangle$ and $\langle f \rangle \neq \langle v \rangle$. Since
$\langle u \rangle \cap \langle v \rangle = \langle
|u|~\wedge~|v|\rangle$ we have the nontrivial
$\Theta$-decomposition $f' = |u| \wedge |v|$ (which is a generator
of $\langle f \rangle$).

Conversely, assume that $f' = |u| \wedge |v|$ is a nontrivial
$\Theta$- decomposition for some $f' \sim_{\FF} f$. Then $\langle
f \rangle = \langle f' \rangle = \langle |u| \wedge |v| \rangle =
\langle u \rangle \cap \langle v \rangle$. Since the decomposition
$f' = |u| \wedge |v|$ is nontrivial, we have that $u \not
\sim_{\FF} f'$ and $v \not \sim_{\FF} f'$, and thus $\langle |u|
\rangle = \langle u \rangle \neq \langle f' \rangle = \langle f
\rangle$. Similarly, $\langle v \rangle \neq \langle f \rangle$.
Thus, by definition, $\langle f \rangle$ is reducible.
\end{proof}

\begin{flushleft} We can equivalently rephrase Lemma \ref{lem_reducibility_and_decomposition} as follows: \end{flushleft}
\begin{rem}
$f$ is reducible if and only if   some $f' \sim_{\FF} f$ has a
nontrivial $\Theta$-decomposition.
\end{rem}

A  question immediately arising from Definition
\ref{defn_decomposition} and
Lemma~\ref{lem_reducibility_and_decomposition} is:

If $f \in F(\Lambda)$ has a nontrivial $\Theta$-decomposition and
$g \sim_{\FF} f$, does $g$ also have a nontrivial
$\Theta$-decomposition? If so, how is this pair of decompositions
related?

In the next few paragraphs we provide an answer to both of these
questions, for  $\Theta = \PCon(\langle \mathscr{R} \rangle)$.

\begin{rem}
$\sum_{i=1}^{k}s_i(a_i \wedge b_i)^{d(i)} =
\left(\sum_{i=1}^{k}s_i a_i^{d(i)} \right) \wedge
\left(\sum_{i=1}^{k}s_i b_i^{d(i)}\right),$ $ \forall s_1, ...,
s_k , a_1,....,a_k  , b_1,....,b_k \in F(\Lambda)$,  and  $d(i)
\in \mathbb{N}_{\geq 0}$.
\end{rem}

\begin{rem}
If $h_1,...,h_k \in F(\Lambda)$   such that each $h_i \ge_{\nu}
1$, then $\sum_{i=1}^{k}s_ih_i \ge_{\nu}  1 $ for every
$s_1,...,s_k \in F(\Lambda)$ such that $\sum_{i=1}^{k}s_i \nucong
1$.
\end{rem}

\begin{thm}\label{thm_every_gen_is_reducible}
(For $\Theta = \PCon(\langle \mathscr{R} \rangle)$.) If $\langle f
\rangle$ is a (principal) reducible kernel, then there exist
$\Theta$-elements $g,h \in F(\Lambda)$ such that $|f| = |g| \wedge
|h|$ and  $|f| \not \sim_{\FF} |g|,|h|$.
\end{thm}

\begin{proof}
If $\langle f \rangle$ is a principal reducible kernel, then there exists $f' \sim_{\FF} f$
such that $f' = |u| \wedge |v| = \min(|u|,|v|)$ for $\Theta$-elements $u,v \in \langle \mathscr{R} \rangle $
with $f'~\not\sim_{\FF}~|u| , |v|$. Then $|f| \in \langle f' \rangle$ since $f'$ is a generator of $\langle f \rangle$,
 so there exist   $s_1,...,s_k \in F(\Lambda)$ such that $\sum_{i=1}^{k}s_i = 1$
 and $|f| = \sum_{i=1}^{k}s_i (f')^{d(i)}$ with $d(i) \in \mathbb{N}_{\geq 0}$. ($d(i) \geq 0$ since $|f| \ge_{\nu}
 1$.)
Thus $$f = \sum_{i=1}^{k}s_i (|u| \wedge |v|)^{d(i)} =
\sum_{i=1}^{k}s_i (\min(|u|,|v|))^{d(i)} =
\min\left(\sum_{i=1}^{k}s_i|u|^{d(i)},\sum_{i=1}^{k}s_i|v|^{d(i)}\right)
=|g| \wedge |h|$$ where $g = |g| = \sum_{i=1}^{k}s_i|u|^{d(i)}, h
= |h| = \sum_{i=1}^{k}s_i|v|^{d(i)}$.

Now   $\langle |f| \rangle \subseteq \langle |g| \rangle \subseteq
\langle |u| \rangle$  and $\langle |f| \rangle \subseteq  \langle
|h| \rangle \subseteq \langle |v| \rangle$, implying $\Skel(f)
\supseteq \Skel(g) \supseteq \Skel(u)$ and $\Skel(f) \supseteq
\Skel(h) \supseteq \Skel(v)$.

We claim that $|g|$ and $|h|$ generate $\langle |u| \rangle $ and
$\langle |v| \rangle$, respectively. Indeed, $\Skel(f')
=\Skel(|f|)$, since $f' \sim_{\FF} |f|$ and thus for any $\bfa \in
F^{(n)}$, $f'(\bfa)=1 \Leftrightarrow |f|(\bfa)=1$. Let $s_j
(f')^{d(j)}$ be a dominant term of $|f|$ at $\bfa$, i.e.,
$$|f| \nucong \sum_{i=1}^{k}s_i(\bfa) (f'(\bfa))^{d(i)} \nucong s_j(\bfa) (f'(\bfa))^{d(j)}.$$
Then   $f(\bfa) \nucong 1 \Leftrightarrow s_j(\bfa)
(f'(\bfa))^{d(j)}\nucong1$. If $f'(\bfa) \nucong 1 $, then
$(f'(\bfa))^{d(j)}\nucong1$, so  $s_j(\bfa) \nucong 1$. Now, for
$\bfa \in \Skel(g)$. Then we have $$g(\bfa) \nucong
\sum_{i=1}^{k}s_i|u|^{d(i)} \nucong 1.$$ Let $s_t|u|^{d(t)}$ be a
dominant term of $g$ at $\bfa$. If $s_t(\bfa) \nucong1$ then
$|u|^{d(t)} \nucong 1$ and thus $u = 1$, and $\bfa \in \Skel(u)$.
Otherwise $s_t(\bfa) <_\nu 1$ (since $\sum_{i=1}^{k}s_i \nucong
1$) and so, by the above, $s_t (f')^{d(t)}$ is not a dominant term
of $|f|$ at $\bfa$. Thus,  for every index j of a dominant term of
$|f|$ at $\bfa$, we have $j \ne t$ and
$$|u(\bfa)|^{d(j)} \nucong s_j(\bfa)|u(\bfa)|^{d(j)} < _\nu s_t(\bfa)|u(\bfa)|^{d(t)} \nucong g(\bfa) \nucong 1.$$

\begin{equation}\label{eq_1_thm_every_gen_is_reducible}
s_j(\bfa)(f'(\bfa))^{d(j)}\nucong s_j(\bfa) (|u|(\bfa) \wedge
|v|(\bfa))^{d(j)} \leq s_j(\bfa)|u(\bfa)|^{d(j)} <_\nu 1.
\end{equation}
On the other hand, $f'(\bfa)\nucong1$ since $\Skel(f) \supseteq
\Skel(g)$, implying $s_j(\bfa)(f'(\bfa))^{d(j)}\nucong1$,
contradicting \eqref{eq_1_thm_every_gen_is_reducible}. Hence,
$\Skel(g) \subseteq \Skel(u)$, yielding $\Skel(g) = \Skel(u)$,
which implies that $g$ is a generator of $\langle |u| \rangle =
\langle u \rangle$. The proofs for $h$ and $|v|$ are analogous.

Consequently,   $g \sim_{\FF} |g| \sim_{\FF} |u|$ and $h
\sim_{\FF} |h| \sim_{\FF} |v|$. Since $|f|~\sim_{\FF}~f' \not
\sim_{\FF} |u| , |v|$ we conclude that $|f| \not \sim_{\FF}
|g|,|h|$.
\end{proof}

\begin{cor}\label{cor_generator_structure}
For $f \in \langle F \rangle$, if $|f| = \bigwedge_{i=1}^{s}
|f_i|$  for   $f_i \in \langle F \rangle$, then for any $g
\sim_{\FF} f$, we have $|g| = \bigwedge_{i=1}^{s} |g_i|$, with
$g_i \sim_{\FF}  f_i$ for $i = 1,...,s$.
\end{cor}
\begin{proof}
Iterate Theorem \ref{thm_every_gen_is_reducible}.
\end{proof}
%

%

\ \\

\begin{cor}\label{cor_norm_decomposition}
If $\langle f \rangle$ is a kernel in $\Theta$, then $\langle f
\rangle$ has a nontrivial decomposition $\langle f \rangle =
\langle g \rangle \cap \langle h \rangle$ if and only if $|f|$ has
a nontrivial decomposition $|f|~=~|g'|~\wedge~|h'|$ with $|g'|
\sim_{\FF} g$ and $|h'| \sim_{\FF} h$.
\end{cor}
\begin{proof}
If $|f| = |g'| \wedge |h'|$  then, since $|g'| \sim_{\FF} g$ and
$|h'| \sim_{\FF} h$ we have $$\langle f \rangle  = \langle |f|
\rangle =  \langle |g'| \wedge |h'| \rangle = \langle |g'| \rangle
\cap \langle  |h'| \rangle = \langle g \rangle \cap \langle h
\rangle.$$ The converse is seen as in  the proof of
Theorem~\ref{thm_every_gen_is_reducible}.
\end{proof}

Corollary \ref{cor_norm_decomposition} provides a
$\Theta$-decomposition of $|f|$, for every generator $f$ of a
reducible kernel in $\Theta$.

\begin{rem}\label{rem_wedge_equivalence}
By \cite[Corollary~4.1.25]{Kern},
$$\langle f \rangle \cap \langle g \rangle = \langle (f + f^*) \wedge (g + g^*)\rangle  = \langle |f| \wedge |g| \rangle.$$
But, in fact, $\langle f \rangle \cap \langle g \rangle = \langle
f' \rangle \cap \langle g' \rangle$  for any  $g' \sim_{\FF} g$
and $h' \sim_{\FF} h$, so we could   take  $|g'| \wedge |f'|$
instead of $|g| \wedge |f|$ on the righthand side of the equality,
e.g., $\langle |f^k| \wedge |g^m| \rangle$ for any $m,k \in
\mathbb{Z} \setminus \{ 0 \}$.
\end{rem}

\begin{defn}
Let $\mathbb{S}$ be a semifield and let $a,b \in \mathbb{S}$. We
say that $a$ and $b$ are $\FF$-\textbf{comparable} if  there exist
some $a' \sim_\FF a$ and $b' \sim_\FF b$ such that  $|a'| \leq |
b'| $ or $|b'| \leq |a'|$.
\end{defn}

Since $|g| \wedge |h| = \min(|g|,|h|)$ we can utilize Remark
\ref{rem_wedge_equivalence} to get the following observation:

\begin{prop}
A $\Theta$-decomposition $f  \sim_{\FF} |g| \wedge |h| \in
F(\Lambda)$ is nontrivial if and only if the $\Theta$-elements $g$
and $h$ are not $\FF$-comparable.
\end{prop}
\begin{proof}
If $g$ and $h$ are $\FF$-comparable, then there exist some $g'
\sim_{\FF} g$ and $h' \sim_{\FF} h$ such that $|g'| \ge_{\nu}
|h'|$ or $|h'| \ge_{\nu}  |g'|$. Without loss of generality,
assume that $|g'| \ge_{\nu}  |h'|$. Then $\langle  |g| \wedge |h|
\rangle  = \langle |g| \rangle \cap \langle |h| \rangle = \langle
|g'| \rangle \cap \langle |h'| \rangle = \langle  |g'| \wedge |h'|
\rangle = \langle \min(|g'|,|h')| \rangle = \langle  |g'| \rangle
= \langle g' \rangle = \langle  g \rangle$. Thus $\langle f
\rangle = \langle g \rangle$ so $f \sim_{\FF} g$ yielding that the
decomposition is trivial.

Conversely, if  $g$ and $h$ are not $\FF$-comparable then we claim
that $f \not \sim_{\FF} g$ and $f \not \sim_{\FF} h$. We must show
that $\langle h \rangle \not \subseteq \langle g \rangle$ and
$\langle g \rangle \not \subseteq \langle h \rangle$ respectively.
So assume that $\langle h \rangle  \supseteq \langle g \rangle$.
In view of Lemma~\ref{prop_HP_element_is_a_generator} and
\cite[Proposition~4.1.13]{Kern}, there exists some $f' \sim f$
such that $f' = |h|^{k} \wedge g$. Note that $|h|^{k} \geq 1$ and
$g \geq 1$ so $f' = |h|^{k} \wedge g \geq 1$ and thus $|f'| = f'$.
Finally, $$|f'| = |h|^{k} \wedge g \Leftrightarrow |f'| \leq
|h|^{k} \Leftrightarrow |f'| \in \langle |h| \rangle
\Leftrightarrow f' \in \langle h \rangle \Leftrightarrow f \in
\langle h \rangle.$$
\end{proof}

\subsection{Convex dependence}$ $


\begin{defn}\label{defn_convex_dep}
An  HS-fraction  $f$ of $F(\Lambda)$ is
\textbf{$F$-convexly dependent} on a set $A$   of HS-fractions  if
\begin{equation}\label{eq_defn_convex_dep}
f \in \left\langle \{g : g \in A \} \right\rangle \cdot \langle F \rangle;
\end{equation}
otherwise $f$ is said to be \textbf{$F$-convexly-independent} of
$A$. The set $A$ is said to be $F$-\textbf{convexly independent}
if   $f$  is $F$-convexly independent of $A \setminus \{ f \}$,
for  every $f \in A$. If $\{a_1,...,a_n \}$ is $F$-convexly
dependent, then we also say that $a_1,...,a_n$ are $F$-convexly
dependent.

\end{defn}

Note that under the assumption that $g \in  \langle F \rangle
\setminus \{1\}$ for some $g \in A$, the condition in
\eqref{eq_defn_convex_dep} simplifies to $f \in \left\langle \{g :
g \in A \} \right\rangle $.

\begin{rem}
By  definition, an HS-fraction $f$ is
$F$-convexly   dependent on HS-fractions $\{g_1,...,g_t\} $  if and only if
$$\langle |f|  \rangle = \langle f  \rangle  \subseteq  \langle g_1,..., g_t \rangle \cdot
\langle F \rangle = \Big\langle \sum_{i=1}^{t} |g_i| \Big\rangle \cdot \langle F \rangle
 = \Big\langle \sum_{i=1}^{t} |g_i|  +  |\alpha| \Big\rangle,$$
  for any element $\alpha$ of $F$ for which $\a^\nu \ne \onenu.$
\end{rem}

\begin{exmp}
For any $\alpha \in F$ and any $f \in F(\Lambda)$,
$$|\alpha f| \leq |f|^2  + |\alpha|^2 =(|f|  +  |\alpha|)^2.$$
Thus $\alpha f \in \left\langle (|f|  +  |\alpha|)^2 \right\rangle = \langle |f|  +  |\alpha| \rangle = \langle f \rangle \cdot \langle F \rangle$.
In particular, if $f$ is an HS-fraction, then $\alpha f$ is $F$-convexly dependent on $f$.
\end{exmp}

As a consequence of Lemma~\ref{prop_HP_element_is_a_generator}, if
two $\scrL$-monomials $f,g$, satisfy $g \in \langle f \rangle$,
then $\langle g \rangle = \langle f \rangle$. In other words,
either $\langle g \rangle = \langle f \rangle$ or $\langle g
\rangle \not \subseteq \langle f \rangle$ and $\langle f \rangle
\not \subseteq \langle g \rangle$. This motivates us to restrict
the convex dependence relation to the set of $\scrL$-monomials.
This will be justified later by showing that for each $F$-convexly
independent subset of HS-fractions of order $t$ in $F(\Lambda)$,
there exists an $F$-convexly independent subset of
$\scrL$-monomials having order $ \geq t$ in $F(\Lambda)$. Let us
see that convex-dependent is an abstract dependence relation.

\begin{prop}\label{prop_abstract_dependence_properties}

Let $A, A_1 \subset  F(\Lambda)$ be  sets of HS-fractions, and let
$f$ be an HS-fraction.
\begin{enumerate}
  \item If $f \in A$, then $f$ is $F$-convexly-dependent on $A$.
  \item If $f$ is $F$-convexly dependent on $A$ and   each $a \in A$ is
  $F$-convexly-dependent on $A_1$, then $f$ is $F$-convexly dependent on $A_1$.
  \item If $f$ is $F$-convexly-dependent on $A$, then $f$ is $F$-convexly-dependent on $A_0$ for some finite subset $A_0$ of $A$.
\end{enumerate}
\end{prop}
\begin{proof}$ $
(1)  $f \in \langle A \rangle \subseteq
\langle A \rangle \cdot \langle F \rangle $.

(2)  $\langle A \rangle \subseteq \langle A_1 \rangle \cdot
\langle F \rangle$ since $a$ is convexly-dependent on $A_1$  for
each $a \in A$. If $f$ is $F$-convexly dependent on~$A$, then $f
\in \langle A \rangle \cdot \langle F \rangle \subseteq \langle
A_1 \rangle \cdot \langle F \rangle$, so, $f$ is $F$-convexly
dependent on $A_1$.

 (3) $a \in
\langle A \rangle \cdot \langle F \rangle$, so by
Proposition \ref{prop_ker_stracture_by group} there exist some
$s_1,...,s_k \in F(\Lambda) $ and  $g_1,...,g_k \in G(A \cup
F) \subset \langle A \rangle \cdot \langle F
\rangle$, where $G(A \cup F)$ is the group generated by $A
\cup F$, such that $\sum_{i=1}^{k}s_i = 1$ and $a =
\sum_{i=1}^{k}s_i g_i^{d(i)}$ with $d(i) \in \mathbb{Z}$. Thus $a
\in \langle g_1,...,g_k \rangle$ and  $A_0  = \{g_1,...,g_k\}$.
\end{proof}

From now on, we assume that the $\nu$-\semifield0 $F$ is
divisible.

\begin{prop}[Steinitz exchange axiom]\label{prop_convex_dependence_of_HP_property}
 Let $S =
\{b_1,...,b_t\} \subset \operatorname{HP}(F(\Lambda))$ and let $f$
and $b$ be elements of $\operatorname{HP}(F(\Lambda))$. If $f$ is
$F$-convexly-dependent on $S \cup \{ b \}$ and $f$ is $F$-convexly
independent of $S$, then $b$ is $F$-convexly-dependent on $S \cup \{
f \}$.
\end{prop}
\begin{proof}
We may assume that $\alpha \in S$ for some $\alpha \in F$. Since
$f$ is $F$-convexly independent of $S$, by definition $f \not \in
\langle S \rangle$ this implies that $\langle S \rangle \subset
\langle S \rangle \cdot \langle f \rangle$ (for otherwise $\langle
f \rangle \subseteq \langle S \rangle$ yielding that $f$ is
$F$-convexly dependent on $S$).  Since $f$ is
$F$-convexly-dependent on $S \cup \{ b \}$, we have that  $f \in
\langle S \cup \{ b \} \rangle = \langle S \rangle \cdot \langle b
\rangle$. In particular, we get that $b \not \in \langle S \rangle
\cdot \langle F \rangle$ for otherwise $f$ would be dependent on
$S$. Consider the quotient map $\phi : F(\Lambda)  \rightarrow
F(\Lambda)/\langle S \rangle$. Since $\phi$ is a semifield
epimorphism and $f,b  \not \in \langle S \rangle \cdot \langle F
\rangle = \phi^{-1}(\langle F \rangle)$, we have that $\phi(f)$
and $\phi(b)$ are not in $F$ thus are $\scrL$-monomials in the
semifield $\Im(\phi) =F(\Lambda)/\langle S \rangle$. By the above,
$\phi(f) \neq 1$ and $\phi(f) \in \phi(\langle b \rangle) =
\langle \phi(b) \rangle$. Thus, $\langle \phi(f) \rangle =\langle
\phi(b) \rangle$ by Lemma~\ref{prop_HP_element_is_a_generator}. So
$\langle S \rangle \cdot \langle f \rangle = \phi^{-1}(\langle
\phi(f) \rangle) = \phi^{-1}(\langle \phi(b) \rangle) = \langle S
\rangle \cdot \langle b \rangle$, consequently $b \in \langle S
\rangle \cdot \langle b \rangle = \langle S \rangle \cdot \langle
f \rangle = \langle S \cup \{ f \} \rangle $, i.e., $b$ is
$F$-convexly-dependent on $S \cup \{ f \}$.
\end{proof}

\begin{defn}\label{defn_convex_span_in_semifield_of_fractions}
Let $A \subseteq \operatorname{HP}(F(\Lambda))$. The \textbf{convex
span} of $A$ over $F$ is the set
\begin{equation}
\ConSpan_{F}(A) = \{ a \in \operatorname{HP}(F(\Lambda)) : a \text{
is $F$-convexly dependent on $A$} \}.
\end{equation} For a \semifield0 $\mathrm{K} \subseteq F(\Lambda)$  such that
$F \subseteq \mathrm{K}$.
 a set  $A \subseteq
\operatorname{HP}(F(\Lambda))$  is said to \textbf{convexly span}
$\mathrm{K}$ over $F$ if  $$\operatorname{HP}(\mathrm{K}) =
\ConSpan_{F}(A).$$
\end{defn}

\begin{rem}
$\ConSpan(\{ f_1,...,f_m\}) =  \langle f_1,...,f_m \rangle \cdot \langle F \rangle.$
\end{rem}

In view of Propositions \ref{prop_abstract_dependence_properties}
and \ref{prop_convex_dependence_of_HP_property}, convex dependence
on  $\operatorname{HP}(F(\Lambda))$ is an abstract dependence
relation. Then by \cite[Chapter 6]{Ro}, we have:

\begin{cor}\label{cor_basis_for_HP}
Let $V \subset \operatorname{HP}(F(\Lambda))$. Then $V$ contains a
basis $B_V \subset V$, which is a maximal convexly independent
subset of unique cardinality such that $$\ConSpan(B_V) =
\ConSpan(V).$$
\end{cor}

\begin{exmp}\label{exmp_basis_for_semifield_of_fractions}
By Lemma~\ref{prop_maximal_kernels_in_semifield_of_fractions_part2},
the maximal kernels in $\PCon(F(\Lambda))$ are HS-fractions of the
form $L_{(\alpha_1,...,\alpha_n)}=\langle \alpha_1 x_1, \dots,
\alpha_n x_n \rangle$ for any $\alpha_1,...,\alpha_n \in F$. In view
of Corollary~\ref{gen5},
$$F(\Lambda) = \ConSpan(\{ \alpha_1 x_1, \dots, \alpha_n x_n \}),$$
i.e., $\{ \alpha_1 x_1, \dots, \alpha_n x_n \}$ convexly spans
$F(\Lambda)$ over $F$.  Now $\alpha_k x_k \not \in \langle
\bigcup_{j \neq k}\alpha_j x_j \rangle \cdot~\langle F\rangle$,
since there are no order relations between $\alpha_i x_i$ and the
elements of  $\{ \alpha_j x_j : j \neq i \} \cup \{\alpha\ : \alpha
\in F \}$.  Thus, for arbitrary $\alpha_1,...,\alpha_n \in F$, $\{
\alpha_1 x_1, \dots, \alpha_n x_n \}$ is $F$-convexly independent,
constituting a basis for $F(\Lambda)$.

\begin{defn}
Let $V \subset \operatorname{HP}(F(\Lambda))$ be a set of
$\scrL$-monomials. We define the \textbf{convex  dimension} of $V$,
$\condeg(V)$, to be $|B|$ where $B$ is a basis for $V$.
\end{defn}

\begin{exmp}\label{rem_convexity_degree_of_semifield_of_fractions}
$\condeg\left(F(\Lambda)\right) = n$, by Example \ref{exmp_basis_for_semifield_of_fractions}.
\end{exmp}

\begin{rem}\label{rem_dependence_is_a_lattice}
If $S \subset \operatorname{HP}(F(\Lambda))$, then for any $f, g \in
F(\Lambda)$ such that $f,g \in \ConSpan(S)$
$$|f|  +  |g| \in \ConSpan(S) \ \ \text{and} \ \ |f| \wedge |g| \in \ConSpan(S).$$
\end{rem}
\begin{proof}
First we prove that $|f|  +  |g| \in \ConSpan(S)$. Since
$\langle f \rangle \subseteq \langle S \rangle \cdot \langle
F \rangle$ and $\langle g \rangle \subseteq \langle S
\rangle \cdot \langle F \rangle$, we have $\langle f, g
\rangle =  \langle |f|  +  |g| \rangle = \langle f \rangle
\cdot \langle g \rangle \subseteq \langle S \rangle \cdot \langle
F \rangle$.  $|f| \wedge |g| \in \ConSpan(S)$,  since
$\langle |f| \wedge|g| \rangle = \langle f \rangle \cap \langle g
\rangle \subseteq \langle g \rangle \subseteq \langle  S \rangle
\cdot \langle F \rangle$.
\end{proof}

\begin{rem}\label{rem_HS_kernel_finitely_spanned}
If $K$ is an HS-kernel, then $K$ is generated by an HS-fraction
$f~\in~F(\Lambda)$ of the form $f = \sum_{i=1}^{t} |f_i|$ where
$f_1,...,f_t$ are $\scrL$-monomials. So,
$$\ConSpan(K) = \langle F \rangle \cdot K =\langle F \rangle \cdot \langle f \rangle = \langle F \rangle \cdot \Big\langle \sum_{i=1}^{t} |f_i| \Big\rangle = \langle F \rangle \cdot \prod_{i=1}^{t} \langle f_i \rangle$$ $$= \langle F \rangle \cdot \langle f_1,...,f_t \rangle$$ and so, $\{f_1,..., f_t \}$ convexly spans $\langle F \rangle \cdot K$.
\end{rem}

\begin{rem}\label{rem_HS_depend_on_HP}
Let $f$ be an HS-fraction. Then $f \sim_{\FF} \sum_{i=1}^{t}
|f_{i}|$ where $f_{i}$ are   $\scrL$-monomials. Hence $f$ is
$F$-convexly dependent on $\{ f_1,...,f_t \}$, since $\langle f
\rangle = \prod_{i=1}^{t} \langle f_i \rangle = \langle \{
f_1,...,f_t \} \rangle$.
\end{rem}

\begin{lem}\label{lem_HS_HP_convexity_relations}
 Suppose $\{ b_1, ...,b_m \}$ is a set of HS-fractions,
 such that $b_i\sim_{\FF}\sum_{j=1}^{t_i} |f_{i,j}|$,  where $f_{i,j}$ are $\scrL$-monomials.
 Then $b_1$ is $F$-convexly dependent on $\{ b_2, ...,b_m \}$ if and only if all of its
 summands $f_{1,r}$ for $1 \leq r \leq t_1$ are $F$-convexly dependent on $\{ b_2, ...,b_m \}$.
\end{lem}
\begin{proof}
If $b_1$ is $F$-convexly dependent on $\{ b_2, ...,b_m \}$, then
$$ \prod_{j=1}^{t_1} \langle f_{1,j} \rangle = \Big\langle \sum_{j=1}^{t_1} |f_{1,j}|
\Big\rangle = \langle b_1 \rangle \subseteq \big\langle \{ b_1, ...,b_m \} \big\rangle .$$
Hence $ f_{1,r} \in \prod_{j=1}^{t_1} \langle f_{1,j} \rangle$ is $F$-convexly   dependent on $\{ b_1, ...,b_m \}$ and by
 Remark \ref{rem_HS_depend_on_HP} $f_{1,r}$ is $F$-convexly dependent on  $ \left\{ f_{i,j} : 2 \leq i \leq m ;\ 1 \leq j \leq t_i \right\}$.
Conversely, if each~$f_{1,r}$ is $F$-convexly dependent on $\{
b_2, ...,b_m \}$ for $1 \leq r \leq t_1$, then there exist some
$k_1,...,k_{t_1}$ such that   $|f_{1,r}| \leq
\sum_{i=2}^{m}\sum_{j=1}^{t_i}  |f_{i,j}|^{k_r}$. Hence $b_1$ is
$F$-convexly dependent on $\{ b_2, ...,b_m \}$, by
Corollary~\ref{cor_principal_ker_by_order}.
\end{proof}

\begin{lem}\label{lem_HS_base_gives_rise_to_HP_base}
Let $V = \{ f_1, ...,f_m \}$ be a $F$-convexly independent set of
HS-fractions, with $f_i \sim_{\FF} \sum_{j=1}^{t_i} |f_{i,j}|$
for $\scrL$-monomials $f_{i,j}$. Then there exists an $F$-convexly
independent subset $$S_0 \subseteq S =\left\{ f_{i,j} : 1 \leq i
\leq m ; 1 \leq j \leq t_i \right\}$$ such that $|S_0| \geq |V|$
and $\ConSpan(S_0) = \ConSpan(V)$.
\end{lem}
\begin{proof}
By Remark \ref{rem_HS_depend_on_HP}, \  $f_i$ is dependent on the
set of $\scrL$-monomials $\{ f_{i,j} :  1 \leq j \leq t_i \}
\subset S$ for each    $1 \leq i \leq m$, implying $\ConSpan(S) =
\ConSpan(V)$. By Corollary \ref{cor_basis_for_HP},  $S$ contains a
maximal $F$-convexly independent subset $S_0$ such that
$\ConSpan(S_0) = \ConSpan(S)$ which, by
Lemma~\ref{lem_HS_HP_convexity_relations}, we can shrink down to a
base.
\end{proof}

\begin{lem}\label{rem_HP_not_contains_contained_H}
The following hold for an $\scrL$-monomial $f$:
\begin{enumerate}
  \item $\langle F \rangle \not \subseteq \langle f \rangle$.
  \item If $F(\Lambda)$ is not bounded, then $\langle f  \rangle \not \subseteq \langle F \rangle$.
\end{enumerate}
\end{lem}
\begin{proof} By definition, an $\scrL$-monomial is not bounded from below. Thus
$\langle f \rangle \cap F = \{ 1 \}$, yielding $\langle F \rangle
\not \subseteq \langle f \rangle$. For the second assertion,   an
HP-kernel is not bounded when $F(\Lambda)$ is not bounded, so
$\langle f \rangle \not \subseteq \langle F \rangle$.
\end{proof}

A direct consequence of Lemma~\ref{rem_HP_not_contains_contained_H} is:
\begin{lem}
If $F(\Lambda)$ is not bounded, then any nontrivcial HS-kernel (i.e.,
$\ne \langle 1 \rangle$) is   $F$-convexly independent.
\end{lem}
\begin{proof}
By Lemma~\ref{rem_HP_not_contains_contained_H} the assertion is
true for HP-kernels, and thus for   HS-kernels, since every
HS-kernel contains some HP-kernel.
\end{proof}

\subsection{Computing convex dimension}\label{compcon}$ $

Having justified our restriction to $\scrL$-monomials, we move ahead
with computing lengths of chains.

\begin{rem}\label{rem_basis_for_HS_kernel}
Let $K$ be an HS-kernel of $F(\Lambda)$.
By definition there are   $\scrL$-monomials $f_1,...,f_t\in
\operatorname{HP}(K)$ such that $K = \langle \sum_{i=1}^{t} |f_i|
\rangle$. By Remark \ref{rem_HS_kernel_finitely_spanned},
 $\ConSpan(K)$ is convexly spanned by ${f_1,...,f_t}$. Now,
 since $\ConSpan(K) = \ConSpan(f_1,...,f_t)$ and $\{f_1,...,f_t\} \subset \operatorname{HP}(K) \subset \operatorname{HP}(F(\Lambda))$,
 by Corollary~\ref{cor_basis_for_HP}, $\{f_1,...,f_t\}$ contains a basis $B=\{b_1,...,b_s\} \subset \{f_1,...,f_t\}$ of $F$-convexly independent elements, where $s = \condeg(K),$ such that $\ConSpan(B) = \ConSpan(f_1,...,f_t)=\ConSpan(K)$.
\end{rem}

%

\begin{prop}\label{prop_order_independence_1}
For  any order kernel $o$ of $F(\Lambda)$, if $\scrL$-monomials
$h_1,...,h_t $  are $F$-convexly dependent, then the images of $
h_1,...,h_t $  is   $F$-convexly dependent (in the quotient
\semifield0 $F(\Lambda)/o$).
\end{prop}
\begin{proof}
Denote by $\phi_o: F(\Lambda) \rightarrow F(\Lambda)/o$ the quotient $F$-homomorphism.
Then $\phi_o( \langle F \rangle) =  \langle \phi_o(F) \rangle = \langle F \rangle_{F(\Lambda)/ o}$.
Now, if $  h_1,...,h_t  $ are $F$-convexly dependent then there exist some $j$,
say without loss of generality $j=1$, such that $ h_1 \in \langle h_2,...,h_t \rangle \cdot \langle F \rangle$.
 By assumption and Proposition~\ref{prop_principal_ker},
\begin{align*}
\phi_o( h_1) \ &   \in \phi_o \left(\langle h_2,...,h_t, \alpha \rangle \right)\\
& = \left\langle \phi_o(h_2),...,\phi_o(h_t),\phi_o(\alpha) \right\rangle = \langle\phi_o(h_2),...,\phi_o(h_t), \alpha \rangle\\
& = \langle \phi_o(h_2),...,\phi_o(h_t)  \rangle \cdot \langle F
\rangle
\end{align*}
Thus $\phi_o( h_1)  $ is $F$-convexly dependent on $\{
\phi_o(h_2),...,\phi_o(h_t) \}$.
\end{proof}

\begin{flushleft}Conversely, we have:\end{flushleft}

\begin{lem}\label{lem_order_independence_2}
For any order kernel   $o$  of $F(\Lambda)$, and any   set $\{
h_1,...,h_t \}$  of $\scrL$-monomials, if
$\phi_o(h_1),...,\phi_o(h_t)$ are $F$-convexly dependent in the
quotient \semifield0   $F(\Lambda)/o$  and
$\sum_{i=1}^{t}\phi_o(|h_i|)\cap F = \{1\}$, then
 $  h_1,...,h_t  $   are $F$-convexly dependent   in
$F(\Lambda)$.
\end{lem}
\begin{proof}
Note that $\sum_{i=1}^{t}\phi_o(|h_i|) \cap F = \{1\}$ if and only
if $\bigcap_{i=1}^{t}\skel(h_1) \cap \skel(o) \neq \emptyset$.
Translating the variables by a point $a \in
\bigcap_{i=1}^{t}\skel(h_1) \cap \skel(o)$, we may assume that the
constant coefficient of each $\scrL$-monomial $h_i$ is~$1$. Assume
that $\phi_o(h_1),...,\phi_o(h_t)$ are $F$-convexly dependent. We
may assume that $\phi_o(h_1)$ is $F$-convexly dependent on
$\phi_o(h_2),...,\phi_o(h_t)$.  This means by
Definition~\ref{defn_convex_dep} that we can take $h_{t+1} \in F$
for which $\phi_o(h_1)\in \langle
\phi_o(h_2),...,\phi_o(h_t),\phi_o(h_ {t+1}) \rangle$. Taking the
pre-images of the quotient map yields $$\langle h_1 \rangle \cdot o
\subseteq \langle h_2,...,h_t, h_{t+1} \rangle \cdot o.$$ Take an
$\scrL$-monomial $g$ such that   $1 +  g$ generates $o$. By
Corollary~\ref{cor_principal_ker_by_order}, there exists some $k \in
\mathbb{N}$ such that
\begin{equation}\label{eq_order_independence_2_1}
|h_1|  +  | 1  +  g| \leq_{\nu} (|h_2|  +  \dots  +  |h_{t+1}|  +
|1  +  g|)^{k} = |h_2|^{k}  +  \dots  +  |h_{t+1}|^{k}  +  |1  +
g|^{k}.
\end{equation}
 As $1  +  g \geq_{\nu} 1$ we have that $|1  +  g | \nucong 1  +  g$, and the
 right hand side of Equation~ \eqref{eq_order_independence_2_1} equals $$|h_2|^{k}  +  \dots  +  |h_{t+1}|^{k}  +  (1  +  g)^{k} \nucong |h_2|^{k}  +  \dots  +  |h_{t+1}|^{k}  +  1  +  g^{k} \nucong |h_2|^{k}  +  \dots +  |h_{t+1}|^{k}  +  g^{k}.$$ The last equality is due to the fact that $\sum|h_i|^k \geq_{\nu} 1$ so that $1$ is absorbed. The same argument,
  applied to the left hand side of Equation \eqref{eq_order_independence_2_1}, yields that
\begin{equation}\label{eq_order_independence_2_2}
|h_1|   +  g \leq_{\nu} |h_2|^{k}  +  \dots  +  |h_{t+1}|^{k}  +
g^{k}.
\end{equation}
Assume on the contrary that $h_1$ is $F$-convexly independent of
$\{ h_2 ,...,h_t  \}$. Then $$\langle h_1 \rangle \not \subseteq
\langle h_2,...,h_{t+1} \rangle \nucong \bigg\langle
\sum_{i=2}^{t+1}|h_i| \bigg\rangle.$$ Thus for any $m \in
\mathbb{N}$ there exists some $\bfa_m \in F^{(n)}$ such that
$$|h_1(\bfa_m)| >_{\nu} \bigg|\sum_{i=2}^{t+1}|h_i(\bfa_m)|\bigg|^{m}
\nucong \sum_{i=2}^{t+1}|h_i(\bfa_m)|^m .$$ Thus by equation
\eqref{eq_order_independence_2_2} and the last observation we get
that $$\sum_{i=2}^{t}|h_i(\bfa_m)|^m   +  g(\bfa_m) <_{\nu}
|h_1(\bfa_m)| + g(\bfa_m) \leq_{\nu} \sum_{i=2}^{t}|h_i(\bfa_m)|^k
+ g(\bfa_m)^k,$$ i.e., there exists some fixed $k \in \mathbb{N}$
such that for any $m \in \mathbb{N}$,
\begin{equation}\label{eq_order_independence_2_3}
\sum_{i=2}^{t}|h_i(\bfa_m)|^m  <_{\nu}
\sum_{i=2}^{t}|h_i(\bfa_m)|^k + g(\bfa_m)^k.
\end{equation}
For $m > k$, since  $|\gamma|^k \leq_{\nu} |\gamma|^m$ for any
$\gamma \in F$, we get that $\sum_{i=2}^{t}|h_i(\bfa_m)|^m
\geq_{\nu} \sum_{i=2}^{t}|h_i(\bfa_m)|^k$. Write $$g^k  = g(1)g'
.$$ Since $g^k$ is an HP-kernel, $g(1)$ is the constant
coefficient of $g$ and $g'$ is a Laurent monomial with coefficient
$1$.

According to the way $\bfa_m$ were chosen,
$\sum_{i=2}^{t}|h_i(\bfa_m)|
> 1$ and $\sum_{i=2}^{t}|h_i(\bfa_m)|^m <_{\nu} g(\bfa_m)^k$, and thus
  $g(1) <_{\nu} g'(\bfa_{m_0})$ for  large enough $m_0$. But
$g'(\bfa_{m}^{-1}) \nucong g'(\bfa_{m})^{-1}$ so
$$g^k(\bfa_{m_0}^{-1}) \nucong g(1)g'(\bfa_{m_0}^{-1}) \nucong
g(1)g'(\bfa_{m_0})^{-1} <_{\nu} 1.$$ Thus
\eqref{eq_order_independence_2_3} yields
$\sum_{i=2}^{t}|h_i(\bfa_{m_0}^{-1})|^m <_{\nu}
\sum_{i=2}^{t}|h_i(\bfa_{m_0}^{-1})|^k$, a contradiction.
\end{proof}

%

\begin{prop}\label{prop_order_independence_2}
Let $R$ be a region kernel of $F(\Lambda)$.  Let $\{ h_1,...,h_t \}$
be a set of $\scrL$-monomials such that $(R \cdot \langle h_1,
...,h_t \rangle) \cap F = \{1 \}$. Then $h_1 \cdot R,...,h_t \cdot
R$ are   $F$-convexly dependent in the quotient \semifield0
$F(\Lambda)/R$ if and only if $h_1,...,h_t$ are $F$-convexly
dependent in $F(\Lambda)$.
\end{prop}
\begin{proof}
The `if' part of the assertion follows from Proposition
\ref{prop_order_independence_1}. Since $R = \prod_{i=1}^{m}o_i$
for suitable order kernels $\{ o_i \}_{i=1}^{m}$, the `only if'
part follows from Lemma \ref{lem_order_independence_2} applied
repeatedly to each of these $o_i$'s.
\end{proof}
%

\begin{prop}\label{prop_convexity_degree_invariance_for_HS_kernels}
Let $R \in \PCon(F(\Lambda))$ be a region kernel. Then,  for any
set $L$ of~HS-fractions, $$\condeg(L) = \condeg(L \cdot R),$$  the
right side taken in $F(\Lambda)/R$.
\end{prop}

\begin{proof}   $\condeg(L) \leq  \condeg(R \cdot L)$ since $L \subseteq  R
\cdot L$. For the reverse inequality, let $\phi_{R} : F(\Lambda)
\rightarrow F(\Lambda)/R$ be the quotient map. Since $L$ is a
sub-\semifield0 of $F(\Lambda)$, $\phi_{R}^{-1}(\phi_R(L))=R \cdot
\langle L \rangle$ by Theorem \ref{thm_nother_1_and_3}, and
$\condeg(L) \geq \condeg(\phi_R(L))$ by Proposition
\ref{prop_order_independence_1}, while  $\condeg(\phi_R(L)) \geq
\condeg(\phi_{R}^{-1}(\phi_R(L)))$ by Lemma
\ref{lem_order_independence_2}. Thus
  $\condeg(L) \geq  \condeg(R \cdot L)$.
\end{proof}

\end{exmp}

%
%
%

In this way we see that $K \mapsto \Omega(K)$ yields a homomorphism
of kernels. Hence $\Omega$ is a natural map in the sense of
Definition~\ref{natural}, and we can apply Theorem~\ref{Schreier}.

\begin{rem}\label{rem_condeg_not_effected_by_region}
Let $R$ be a region kernel
and let
$$A = \{ \langle g \rangle : \langle g \rangle \cdot \langle F
\rangle \supseteq  R \cdot \langle F \rangle \}.$$ Then
$\condeg(F(\Lambda)/R) = \condeg(A)$, in view of  Remark
\ref{rem_correspondence_of_quontient_HSpec} and  Proposition
\ref{prop_order_independence_2}. As $L_{a} \in A$ for any $a \in
\skel(R) \neq \emptyset$
and $\condeg(L_a) = \condeg(F(\Lambda)),$ we conclude that $\condeg(A) = \condeg(F(\Lambda))$.\\
\end{rem}

We are ready for  catenarity of $\condeg.$

\begin{thm}\label{caten}
If $R$ is a region kernel and $L$ is an HS-kernel  of
$F(\Lambda)$, then
$$\condeg(F(\Lambda)/LR) = \condeg(F(\Lambda)) - \condeg(L).$$
In particular, $$\condeg(F(\Lambda)/LR) = n -
\condeg(L).$$
\end{thm}
\begin{proof}
 $F(\Lambda)/LR \cong (F(\Lambda)/R)/ (L\cdot R/R)$, by the third isomorphism theorem.
 Choose a basis for HP$(F(\Lambda)/R)$ containing a basis for HP$(L\cdot R/R)$.
 Then  $\condeg(L \cdot R/R) = \condeg(\phi_R(L))$, by Remark \ref{rem_condeg_not_effected_by_region}.
But $L \cdot R \cap F = \{ 1 \}$. Hence, by Proposition \ref{prop_order_independence_2},
 $\condeg(\phi_R(L)) = \condeg(L)$. So $$\condeg(F(\Lambda)/LR) = \condeg(A)-\condeg(L) = \condeg(F(\Lambda)) -
 \condeg(L).$$

Thus, $$\condeg(F(\Lambda)/LR) =
\condeg(A)-\condeg(L) = n - \condeg(L).$$
\end{proof}

\begin{prop}\label{prop_kernel_descending_chain}
Let $L$ be an HS-kernel in $F(\Lambda)$ with
$\skel(L)\neq\emptyset$. Let $\{ h_1, ...,h_t \}$ be a set of
$\scrL$-monomials in $\HSpec(F(\Lambda))$ such that $\ConSpan(h_1,
...,h_t) = \ConSpan(L)$ and let $L_i = \langle h_i \rangle$. Then
the chain

\begin{equation}\label{eq_desc_chain_kernels}
L = \prod_{i=1}^{u}L_i  \supseteq \prod_{i=1}^{u-1}L_i  \supseteq \dots \supseteq L_1  \supseteq \langle 1 \rangle .
\end{equation}
of HS-kernels is  strictly descending   if and only if
$h_1,....,h_u$ are $F$-convexly independent.
\end{prop}
\begin{proof}
 $(\Rightarrow)$ If $h_u$ is
$F$-convexly dependent on $\{h_1,....,h_{u-1}\}$, then   $L_u =
\langle h_u \rangle  \subseteq \prod_{i=1}^{u-1}L_i \cdot \langle F
\rangle$. Assume that $L_u = \langle h_u \rangle  \not \subseteq
\prod_{i=1}^{u-1}L_i$. Then $\langle F \rangle \subseteq
\prod_{i=1}^{u}L_i$, implying that $\prod_{i=1}^{u}L_i$ is not an
HS-kernel. Thus $L_u = \langle h_u \rangle   \subseteq
\prod_{i=1}^{u-1}L_i$, and the chain is not strictly descending.

  $(\Leftarrow)$  $\skel(\prod_{i=1}^{t}L_i )
\subseteq \skel(L) \neq \emptyset$ for every $0 \leq t \leq u$,
implying  that $(\prod_{i=1}^{t}L_i ) \cap F = \{ 1 \}$ for every $0
\leq t \leq u$ (for otherwise
$\skel(\prod_{i=1}^{t}L_i)~=~\emptyset$). If $\{h_1,....,h_u\}$ is
$F$-convexly independent then  $L_u = \langle h_u \rangle \not
\subseteq \prod_{i=1}^{u-1}L_i \cdot \langle F \rangle$. By
induction, the
 chain \eqref{eq_desc_chain_kernels}  is strictly descending.

\end{proof}

\begin{thm}\label{prop_kernel_descending_chain_properties}
If $L \in \HSpec(F(\Lambda))$, then $\hgt(L) = \condeg(L)$,
cf.~Definition~\ref{height0}.
 Moreover, every factor of a  descending chain of HS-kernels of maximal length is an  HP-kernel.
\end{thm}
\begin{proof}
By Proposition \ref{prop_kernel_descending_chain},   the maximal
length of a chain of   HS-kernels descending from an HS-kernel $L$
equals the number of elements in a basis of $\ConSpan(L)$; thus  the
chain is of unique length $\condeg(L)$, i.e., $\hgt(L) =
\condeg(L)$. Moreover, by Theorem \ref{thm_nother_1_and_3}(2),
$$\prod_{i=1}^{j}L_i / \prod_{i=1}^{j-1}L_i   \cong L_j/\bigg(L_j \cap \prod_{i=1}^{j-1}L_i \bigg).$$
Furthermore $$\left(L_j \cdot (L_j \cap (\prod_{i=1}^{j-1}L_i  ))\right) \cap F = \{1\},$$
since $L_j \cdot (L_j \cap \prod_{i=1}^{j-1}L_i  ) = L_j \cap \prod_{i=1}^{j}L_i   \subset \prod_{i=1}^{j}L_i $
 and $(\prod_{i=1}^{j}L_i ) \cap F = \{1\}$.
 So the  image of the HP-kernel~$L_j$ in $F(\Lambda)/(L_j \cap (\prod_{i=1}^{j-1}L_i))$  is an HP-kernel. Thus, every factor of the chain is an HP-kernel.
\end{proof}

\begin{cor}\label{correctdim}
$\Hdim(F(\Lambda)) = \condeg(F(\Lambda))= n.$
\end{cor}
\

\begin{rem}\label{cover_of_regions_sub_direct_product}
If $R_1 \cap R_2 \cap \dots \cap R_t = \{1 \}$, then   $F(\Lambda)$ is a subdirect product
$$F(\Lambda) = F(\Lambda)/(R_1 \cap R_2 \cap \dots \cap R_t) \hookrightarrow \prod_{i=1}^{t} F(\Lambda)/R_i.$$
Then for  any kernel $K$ of $F(\Lambda)$, $R_1 \cap R_2 \cap \dots
\cap R_t \cap K = \bigcap_{i=1}^t (R_i \cap K) = \{1 \}$ and, since
$K$ itself is an idempotent \semifield0,
$$K = K/\bigcap_{i=1}^t (R_i \cap K) \cong \prod_{i=1}^{t} K/(R_i \cap K) \cong \prod_{i=1}^{t} R_i K/R_i.$$
\end{rem}

\subsection{Summary}$ $

In conclusion, for every principal regular kernel $\langle f \rangle
\in P(F(\Lambda))$, we have  obtained explicit region kernels
$\{R_{1,1},...,R_{1,s},R_{2,1},...,R_{2,t} \}$ having trivial
intersection, such that
$$\langle f \rangle = \bigcap_{i=1}^{s}K_i \cap \bigcap_{j=1}^{t}N_j$$
where $K_i = L_i \cdot R_{1,i}$ for $i =1,...,s$ and appropriate
HS-kernels $L_i$ and $N_j = B_j \cdot R_{2,j}$ for $j =1,...,t$
and appropriate bounded from below kernels $B_j$. If $\langle f
\rangle \in \PCon(\langle F \rangle)$,  then, in view of
Theorem~\ref{thm_HP_expansion} we can take $B_j = \langle F
\rangle$ for every $j = 1,...,t$. Note that over the various
regions in $F^{(n)}$ corresponding to the region
kernels~$R_{i,j}$, $f$ is locally represented by distinct
HS-fractions in $\HSpec(F(\Lambda))$. In fact each region is
defined so
 that the local HS-representation of $f$ is given over the entire region.
 Thus the    $R_{i,j}$'s  defining the partition of the space
 can be obtained as a minimal set of regions over each of which $\langle f \rangle$
takes the form of an HS-kernel.

For each $j = 1,...,t$,  $\condeg(N_j) = \Hdim(N_j) = 0$,  since
$N_j$ contains no elements of $\operatorname{HP}(F(\Lambda))$,
implying
$$\condeg(F(\Lambda)/N_j) = \Hdim(F(\Lambda)/N_j) = n.$$ For each $i
= 1,...,s$,
 $\condeg(K_i)= \condeg(L_i) = \Hdim(L_i) \geq 1$, implying
$$\condeg(F(\Lambda)/K_i) =
\Hdim(F(\Lambda)/K_i) = n - \Hdim(L_i) < n.$$

%
\begin{rem}
In view of the discussion in \cite[\S 9.2]{Kern}, each term
$F(\Lambda)/L_i$ corresponds to the linear subspace of~$F^{(n)}$
(in logarithmic scale) defined by the linear constraints endowed
on the quotient $F(\Lambda)/L_i$ by the HS-kernel~$L_i$. One can
think of these terms as an algebraic description of the affine
subspaces
 locally comprising  $\skel(f)$.
\end{rem}

\clearpage

\bibliographystyle{amsplain}

\providecommand{\bysame}{\leavevmode\hbox
to3em{\hrulefill}\thinspace}
\providecommand{\MR}{\relax\ifhmode\unskip\space\fi MR }
\providecommand{\MRhref}[2]{%
  \href{http://www.ams.org/mathscinet-getitem?mr=#1}{#2}
} \providecommand{\href}[2]{#2}

\end{document}